\documentclass[10pt,a4paper, reqno]{amsart}
\usepackage{cmap}  
\usepackage[T1]{fontenc}
\usepackage[utf8]{inputenc}
\usepackage[english]{babel}
\usepackage{amsmath,amssymb,mathtools}
\usepackage{enumitem}
\usepackage{microtype}
\usepackage[a4paper,margin=1.65cm]{geometry}
\usepackage[
    colorlinks=true,
    linkcolor=black,
    citecolor=blue,
    urlcolor=black
]{hyperref}
\usepackage{lmodern}
\renewcommand{\uppercasenonmath}[1]{}

\newcommand{\Char}{\operatorname{char}}
\newcommand{\id}{\operatorname{id}}
\newcommand{\rad}{\operatorname{rad}}

\newcommand{\Barun}{\operatorname{Bar}}
\newcommand{\Bbar}{\overline{\operatorname{Bar}}}
\newcommand{\HH}{\operatorname{HH}}
\theoremstyle{definition}
\newtheorem{proposition}{Proposition}[section]
\numberwithin{equation}{section}
\newtheorem{theorem}[proposition]{Theorem}
\newtheorem{definition}[proposition]{Definition}
\newtheorem{lemma}[proposition]{Lemma}
\newtheorem{corollary}[proposition]{Corollary}
\newtheorem{remark}[proposition]{Remark}

\title{BV structure on the Hochschild cohomology of Liu--Schulz algebras in characteristic two}
\author{Artur Tursunbaev}

\address{St.~Petersburg State University,
Universitetskii pr.~28,
St.~Petersburg 198504, Russia}

\email{\href{mailto:tursunbaev.art@gmail.com}{\textcolor{blue}{tursunbaev.art@gmail.com}}}


\newcommand{\parttitle}[1]{%
  \bigskip
  \begin{center}\large\bfseries #1\end{center}
  \nopagebreak\medskip}
  
\begin{document}

\begin{abstract}
Let $K$ be a field of characteristic $2$, let $\rho\in K^\times$, and let
$R_\rho$ be the corresponding Liu--Schulz algebra. We give a complete description
of the Gerstenhaber bracket and the BV structure on the Hochschild cohomology ring
$\HH^*(R_\rho)$.
\end{abstract}

\maketitle

\section{Introduction}\label{sec:intro}

Let $A$ be an associative algebra over a field $K$. The Hochschild
cohomology $\HH^*(A)$ carries a rich structure: the cup
product makes it a graded commutative algebra, and the graded
Lie bracket of degree $-1$ makes it a
Gerstenhaber algebra~\cite{Ger63}.

The Hochschild cohomology of algebras in certain special classes
carries an additional structure, called a BV structure. It is defined
by an operator $\Delta$ of degree $-1$ with square zero which, together with the
cup product, expresses the Lie bracket. Tradler proved that
the Hochschild cohomology algebra of a finite-dimensional symmetric algebra
is a BV algebra~\cite{Tr08}: a nondegenerate symmetric
associative form defines a duality between Hochschild cohomology and
homology, under which the Connes operator corresponds to $\Delta$.

The practical value of the BV structure is that it reduces the computation of the Gerstenhaber bracket, whose direct computation is often laborious, to the computation of the operator \(\Delta\).

This paper concerns Liu--Schulz algebras.
Let $K$ be a field of characteristic $2$, let $\rho\in K^\times$,
and let $R_\rho$ be the corresponding eight-dimensional
local symmetric algebra.
An additive basis of $\HH^*(R_\rho)$ and a formula for the product
were obtained by Generalov and Kosovskaya~\cite{GK06}
using the minimal projective bimodule
resolution $Q_*$ constructed by Kosovskaya~\cite{Kos05}.

For $\rho\ne1$, the isomorphisms of graded $K$-algebras
asserted in parts~II.2 and~II.3 of Theorem~2.1
in~\cite{GK06} are not valid
(Remark~\ref{rem:GK-presentations}):
the defining relations omit the equalities
$y_0^2=y_1^2=y_2^2=0$ for the generators of degree one.
We prove that adding these relations
gives complete presentations of the ring $\HH^*(R_\rho)$
(Theorem~\ref{thm:ring-presentation}).

The Gerstenhaber bracket and the operator $\Delta$
are defined on the bar resolution, whereas the product formula
was obtained on the minimal resolution. Thus
our computations use explicit comparison
morphisms between the normalized bar resolution
and Kosovskaya's resolution.
They are constructed by a method based
on the notion of a weak self-homotopy~\cite{IIVZ}.

The first part of the paper gives an explicit right
$R_\rho$-linear weak self-homotopy of Kosovskaya's resolution
and constructs comparison morphisms $\Phi_*$ and $\Psi_*$, for which the strict equality $\Psi_*\Phi_*=\id_{Q_*}$ is proved. The second part contains a corrected presentation of the cohomology ring, the computation of the operator $\Delta$
and the Gerstenhaber bracket on the generators, and closed formulas for $\Delta$ on arbitrary products
of generators. The cases $\operatorname{ord}\rho=\infty$ and $1<\operatorname{ord}\rho<\infty$ are treated separately, while the case $\rho=1$ is excluded from the main theorems and considered
in Section~\ref{sec:s1}.

\parttitle{Part I. Weak self-homotopy and comparison morphisms}

\section{Hochschild cohomology}\label{sec:prelim}

The following conventions are used throughout this section:
\begin{itemize}[leftmargin=2em,itemsep=.3em]
\item $K$ is a field, and $A$ is a finite-dimensional associative unital
      $K$-algebra;
\item $A^e=A\otimes_KA^{\operatorname{op}}$ is the enveloping algebra
      of $A$; $A$-bimodules are identified with left
      $A^e$-modules;
\item $DA=\operatorname{Hom}_K(A,K)$ is the dual bimodule with
      structure $(a\varphi b)(x)=\varphi(bxa)$.
\end{itemize}

\subsection{The Hochschild complex}

\begin{definition}\label{def:HH}
For an $A$-bimodule $M$, the Hochschild cohomology groups of $A$ with
coefficients in $M$ are defined by
\[
        \HH^n(A,M)=\operatorname{Ext}^n_{A^e}(A,M),
        \qquad n\ge0 .
\]
\end{definition}

In what follows, we set
\[
        \HH^n(A)=\HH^n(A,A),
        \qquad
        \HH^*(A)=\bigoplus_{n\ge0}\HH^n(A).
\]

These groups are computed using the bar resolution $\Barun_*(A)$, where
$\Barun_r(A)=A^{\otimes(r+2)}$:
\[
        \Barun_*(A):\quad\cdots\to A^{\otimes(r+2)}
        \xrightarrow{\,d_r\,}A^{\otimes(r+1)}\to\cdots\to A^{\otimes2}
        \xrightarrow{\,d_0=\mu\,}A\to0,
\]
where $\mu$ is multiplication in $A$ and
\[
        d_r(a_0\otimes\cdots\otimes a_{r+1})
        =\sum_{i=0}^{r}(-1)^i\,
        a_0\otimes\cdots\otimes a_ia_{i+1}\otimes\cdots\otimes a_{r+1}.
\]

Alongside $\Barun_*(A)$, we use the normalized bar resolution
$\Bbar_*(A)$: set $\overline A=A/(K\cdot1)$,
\[
        \Bbar_r(A)=A\otimes_K\overline A^{\otimes r}\otimes_KA,
\]
and the differential is induced by $d_*$ and, on the generators
of the free $A^e$-module $\Bbar_r(A)$, is given by

\[
b_r\bigl(1\otimes\overline a_1\otimes\cdots\otimes\overline a_r\otimes1\bigr)
=a_1\otimes\overline a_2\otimes\cdots\otimes1
+\sum_{i=1}^{r-1}(-1)^i\,1\otimes\cdots\otimes\overline{a_ia_{i+1}}\otimes\cdots\otimes1
+(-1)^r\,1\otimes\cdots\otimes\overline a_{r-1}\otimes a_r .
\]

Both complexes are $A^e$-projective resolutions of $A$.

Set
\[
        C^*(A)=\operatorname{Hom}_{A^e}\bigl(\Barun_*(A),A\bigr),
        \qquad
        \overline C^*(A)
        =\operatorname{Hom}_{A^e}\bigl(\Bbar_*(A),A\bigr).
\]
The formula
\[
        f\longmapsto\bigl((a_1,\dots,a_n)\mapsto
        f(1\otimes a_1\otimes\cdots\otimes a_n\otimes1)\bigr)
\]
gives the identifications
\[
        C^n(A)\cong\operatorname{Hom}_K(A^{\otimes n},A),
        \qquad
        \overline C^n(A)
        \cong\operatorname{Hom}_K(\overline A^{\otimes n},A),
\]
under which the differential in both cases takes the form
\[
        (\delta^nf)(a_1,\dots,a_{n+1})
        =a_1f(a_2,\dots,a_{n+1})
        +\sum_{i=1}^{n}(-1)^i f(a_1,\dots,a_ia_{i+1},\dots,a_{n+1})
        +(-1)^{n+1}f(a_1,\dots,a_n)a_{n+1}.
\]
We use the normalized complex in what follows.

\subsection{The cup product and the Gerstenhaber bracket}

\begin{definition}\label{def:cup}
For $\alpha\in C^n(A)$ and $\beta\in C^m(A)$, the cup product
$\alpha\smile\beta\in C^{n+m}(A)$ is defined by
\[
        (\alpha\smile\beta)(a_1,\dots,a_{n+m})
        =\alpha(a_1,\dots,a_n)\cdot\beta(a_{n+1},\dots,a_{n+m}).
\]
\end{definition}

The product $\smile$ is compatible with $\delta$ and makes $\HH^*(A)$ a
graded commutative $K$-algebra~\cite[Corollary~1]{Ger63}.

\begin{definition}\label{def:circ}
Let $\alpha\in C^n(A)$ and $\beta\in C^m(A)$. For $n,m\ge1$ and
$1\le i\le n$, the partial insertion $\alpha\circ_i\beta\in C^{n+m-1}(A)$
is defined by
\[
        (\alpha\circ_i\beta)(a_1,\dots,a_{n+m-1})
        =\alpha\bigl(a_1,\dots,a_{i-1},
        \beta(a_i,\dots,a_{i+m-1}),a_{i+m},\dots,a_{n+m-1}\bigr);
\]
for $n\ge1$ and $m=0$, that is, $\beta\in A$,
\[
        (\alpha\circ_i\beta)(a_1,\dots,a_{n-1})
        =\alpha(a_1,\dots,a_{i-1},\beta,a_i,\dots,a_{n-1}),
\]
and in all other cases $\alpha\circ_i\beta=0$. We then set
\[
        \alpha\circ\beta=\sum_{i=1}^{n}
        (-1)^{(m-1)(i-1)}\alpha\circ_i\beta,
        \qquad
        [\alpha,\beta]=\alpha\circ\beta
        -(-1)^{(n-1)(m-1)}\beta\circ\alpha\ \in C^{n+m-1}(A).
\]
\end{definition}

The bracket $[-,-]$ induces a well-defined map
$\HH^n(A)\times\HH^m(A)\to\HH^{n+m-1}(A)$.

\begin{definition}\label{def:Ger}
A \emph{Gerstenhaber algebra} is a triple
$(\mathcal A^\bullet,\smile,[-,-])$ in which
$(\mathcal A^\bullet,\smile)$ is a graded commutative
$K$-algebra and
$[-,-]\colon\mathcal A^n\times\mathcal A^m\to\mathcal A^{n+m-1}$ is a
bilinear map that gives $\mathcal A^{\bullet+1}$
the structure of a graded Lie algebra and satisfies
the Poisson identity
\[
[a,b\smile c]=[a,b]\smile c
+(-1)^{(|a|-1)|b|}b\smile[a,c]
\]
for homogeneous $a,b,c$.
\end{definition}

\begin{theorem}[Gerstenhaber~\cite{Ger63}]
\label{thm:Ger}
For any $K$-algebra $A$, the triple $(\HH^*(A),\smile,[-,-])$ is a
Gerstenhaber algebra.
\end{theorem}

\begin{definition}\label{def:sym}
A finite-dimensional $K$-algebra $A$ is called \emph{symmetric} if
$A\cong DA$ as $A^e$-modules; equivalently, if there exists a
$K$-bilinear form $\langle-,-\rangle\colon A\times A\to K$ that is
associative ($\langle ab,c\rangle=\langle a,bc\rangle$),
symmetric ($\langle a,b\rangle=\langle b,a\rangle$), and
nondegenerate.
\end{definition}

The correspondence between the two descriptions is as follows. An isomorphism
$A\xrightarrow{\ \sim\ }DA$ determines the functional $\tau\in DA$ given by the image
of the identity element, and the form
\[
        \langle a,b\rangle=\tau(ab).
\]
Conversely, the form defines the isomorphism $a\mapsto\langle a,-\rangle$.

\medskip

Hochschild homology is computed using the complex
$C_*(A)=A\otimes_{A^e}\Barun_*(A)$;
the identification
$C_r(A)=A\otimes_{A^e}A^{\otimes(r+2)}
\cong A^{\otimes(r+1)}$
transforms the differential into
\[
        \partial_r(a_0\otimes\cdots\otimes a_r)
        =\sum_{i=0}^{r-1}(-1)^i\,
        a_0\otimes\cdots\otimes a_ia_{i+1}\otimes\cdots\otimes a_r
        +(-1)^r\,a_ra_0\otimes a_1\otimes\cdots\otimes a_{r-1}.
\]

This complex carries the Connes operator
$B\colon C_r(A)\to C_{r+1}(A)$,
\[
B(a_0\otimes\cdots\otimes a_r)
=\sum_{i=0}^{r}(-1)^{ir}\,
1\otimes a_i\otimes\cdots\otimes a_r\otimes a_0\otimes\cdots\otimes a_{i-1}
+\sum_{i=0}^{r}(-1)^{ir}\,
a_i\otimes1\otimes a_{i+1}\otimes\cdots\otimes a_r\otimes a_0\otimes\cdots\otimes a_{i-1}.
\]
The operator $B$ is a chain map,
satisfies $B^2=0$, and induces
$B\colon\HH_r(A)\to\HH_{r+1}(A)$.

Set
\[
\overline C_*(A)=A\otimes_{A^e}\Bbar_*(A).
\]
On the normalized complex $\overline C_*(A)$,
the second sum vanishes because it contains
the factor $\overline 1=0$.

\medskip

The form in Definition~\ref{def:sym} induces a duality between
Hochschild cohomology and homology:
\begin{equation}\label{eq:duality}
\operatorname{Hom}_K\bigl(\overline C_*(A),K\bigr)
\cong\operatorname{Hom}_{A^e}\bigl(\Bbar_*(A),DA\bigr)
\cong\operatorname{Hom}_{A^e}\bigl(\Bbar_*(A),A\bigr)
=\overline C^*(A),
\end{equation}

where the first isomorphism is canonical and the second is induced by the isomorphism
of $A^e$-modules $DA\cong A$.

Under~\eqref{eq:duality}, the Connes operator corresponds to the operator
\[
        \Delta\colon\HH^n(A)\longrightarrow\HH^{n-1}(A),
        \qquad n\ge1;
\]
we set $\Delta=0$ on $\HH^0(A)$.

\begin{definition}[\cite{IIVZ}]\label{def:BV}
A \emph{BV algebra} is a Gerstenhaber algebra
$(\mathcal A^{\bullet},\smile,[-,-])$ together with an operator
$\Delta\colon\mathcal A^{\bullet}\to\mathcal A^{\bullet-1}$ of degree
$-1$ such that $\Delta\circ\Delta=0$ and
\[
        [a,b]=-(-1)^{(|a|-1)|b|}
        \bigl(\Delta(a\smile b)-\Delta(a)\smile b
        -(-1)^{|a|}a\smile\Delta(b)\bigr)
\]
for homogeneous $a,b\in\mathcal A^{\bullet}$.
\end{definition}

\begin{theorem}[Tradler {\cite[Theorem~1]{Tr08}}]\label{thm:Tradler}
Let $A$ be a symmetric algebra with form $\langle-,-\rangle$. Then
$\HH^*(A)$, together with the cup product, the Gerstenhaber bracket, and
the operator $\Delta$, is a BV algebra, and for
$\alpha\in C^n(A)$ the cochain $\Delta(\alpha)\in C^{n-1}(A)$ is determined
by the equality
\[
        \bigl\langle\Delta(\alpha)(a_1,\dots,a_{n-1}),\,a_n
        \bigr\rangle
        =\sum_{i=1}^{n}(-1)^{i(n-1)}
        \bigl\langle\alpha(a_i,\dots,a_n,a_1,\dots,a_{i-1}),\,1
        \bigr\rangle,
        \qquad a_1,\dots,a_n\in A.
\]
The same formula holds on the normalized complex.
\end{theorem}

\section{The algebra $R_\rho$ and its minimal resolution}\label{sec:algebra}

Let \(K\) be a field of characteristic \(2\), and let
\(\rho\in K^\times\). We consider the Liu--Schulz algebra (see~\cite{GK06})
\[
        R=R_\rho=K\langle x_0,x_1,x_2\rangle/I,
\]
where \(I\) is the two-sided ideal generated by
\[
        x_i^2,\qquad x_{i+1}x_i+\rho x_ix_{i+1},
        \qquad i=0,1,2,
\]
with indices taken modulo \(3\). Since \(\Char K=2\), the latter relation can be written as
\[
        x_{i+1}x_i=\rho x_ix_{i+1}.
\]
Set
\[
        a=x_0x_1,
        \qquad b=x_1x_2,
        \qquad c=x_2x_0,
        \qquad w=x_0x_1x_2.
\]
In what follows, we use the $K$-basis of $R$ given by
\[
        B=\{1,x_0,x_1,x_2,a,b,c,w\}.
\]

Set
\[
        \Lambda=R^e=R\otimes_K R^{\operatorname{op}}.
\]
The minimal \(\Lambda\)-projective bimodule resolution of \(R\), constructed by Kosovskaya~\cite{Kos05}, is obtained as the totalization of a nonnegative tricomplex \(B_{\bullet,\bullet,\bullet}\). For all \(i,j,k\geq0\), set
\[
        B_{ijk}=\Lambda.
\]
The three differentials of the tricomplex decrease the first, second, and third coordinates, respectively:
\[
        d' : B_{ijk}\to B_{i-1,j,k},
        \qquad
        d'' : B_{ijk}\to B_{i,j-1,k},
        \qquad
        d''' : B_{ijk}\to B_{i,j,k-1}.
\]
In bimodule notation, they are given by
\[
        d'(1_{ijk})=\rho^j x_0 1_{i-1,j,k}+\rho^k1_{i-1,j,k}x_0,
\]
\[
        d''(1_{ijk})=\rho^k x_1 1_{i,j-1,k}+\rho^i1_{i,j-1,k}x_1,
\]
\[
        d'''(1_{ijk})=\rho^i x_2 1_{i,j,k-1}+\rho^j1_{i,j,k-1}x_2.
\]
If the corresponding coordinate is zero, the corresponding term is absent.

The totalization has the form
\[
        Q_n=\operatorname{Tot}(B)_n
        =\bigoplus_{i+j+k=n}\Lambda e_{ijk},
\]
and, for $n\ge1$, the differential of the total complex is
$d_n=d'+d''+d'''$, where
\[
d'(e_{ijk})=
\begin{cases}
\rho^j x_0e_{i-1,j,k}+\rho^k e_{i-1,j,k}x_0,& i>0,\\
0,& i=0,
\end{cases}
\]
\[
d''(e_{ijk})=
\begin{cases}
\rho^k x_1e_{i,j-1,k}+\rho^i e_{i,j-1,k}x_1,& j>0,\\
0,& j=0,
\end{cases}
\]
\[
d'''(e_{ijk})=
\begin{cases}
\rho^i x_2e_{i,j,k-1}+\rho^j e_{i,j,k-1}x_2,& k>0,\\
0,& k=0.
\end{cases}
\]

For $n=0$, set $d_0=\varepsilon$, where
$\varepsilon\colon Q_0=\Lambda e_{000}\to R$ is the
augmentation, $\varepsilon(a\,e_{000}\,b)=ab$, 

As a right \(R\)-module, \(Q_n\) has basis
\[
        \{\beta e_{ijk}\mid \beta\in B,\ i+j+k=n\}.
\]

The aim of the construction is to give explicit formulas for a right \(R\)-linear weak self-homotopy
\[
        t_n:Q_n\longrightarrow Q_{n+1}
\]
of this resolution. These maps are then used to construct comparison morphisms, transfer Tradler's BV operator to the minimal resolution, and compute the Gerstenhaber bracket on the algebra \(\HH^*(R)\).

\section{Preliminaries on comparison morphisms}\label{sec:scheme}

Let $A$ be a $K$-algebra,
\[
        A^e=A\otimes_KA^{\operatorname{op}},
\]
and let $P_*$ be an $A^e$-projective resolution of $A$. Alongside
$P_*$, we consider the normalized bar resolution $\Bbar_*(A)$ from
Section~\ref{sec:prelim}.

\medskip

Let $B_A\subset A$ be a set whose images form a $K$-basis
of $\overline A$. Then the elements
\[
        1\otimes\overline b_1\otimes\cdots\otimes
        \overline b_n\otimes1,
        \qquad b_1,\dots,b_n\in B_A,
\]
form a basis of $\Bbar_n(A)$ as a free $A^e$-module.

Recall~\cite{BZZ09} that a weak self-homotopy of the complex
\[
        \cdots\longrightarrow C_n\xrightarrow{\ d_n\ }C_{n-1}
        \longrightarrow\cdots\longrightarrow
        C_0\xrightarrow{\ d_0\ }N\longrightarrow 0
\]
is a collection of \(K\)-linear maps \(t_n:C_n\to C_{n+1}\),
\(n\geq 0\), and \(t_{-1}:N\to C_0\), satisfying
\[
        d_0t_{-1}=\id_N,
        \qquad
        t_{n-1}d_n+d_{n+1}t_n=\id_{C_n},\quad n\geq 0.
\]
The complex \(\Bbar_*(A)\), augmented by \(b_0\), admits a right
\(A\)-linear weak self-homotopy \(s_*\):
\[
        s_{-1}(r)=1\otimes r,
        \qquad
        s_n(a_0\otimes\overline a_1\otimes\cdots\otimes\overline a_n
        \otimes 1)
        =1\otimes\overline a_0\otimes\overline a_1\otimes\cdots\otimes
        \overline a_n\otimes 1,
\]
extended by right \(A\)-linearity. If \(a_0\in K\cdot 1\), then
\(\overline a_0=0\) and the right-hand side is zero; hence
\[
        s_{n+1}s_n=0,\qquad n\geq -1.
\]

\begin{proposition}[{\cite[Proposition~2.2]{IIVZ}}]\label{prop:lift}
Let \(P_*\) be a resolution of a module \(M\) by free modules
with fixed bases \(\{e_{n,i}\}_{i\in X_n}\subset P_n\), let
\(C_*\) be a resolution of a module \(N\) equipped with a weak self-homotopy
\(t_*\), and let \(f:M\to N\) be a module homomorphism. Set
\[
        f_0(e_{0,i})=t_{-1}f d_0^{P}(e_{0,i}),
        \qquad
        f_n(e_{n,i})=t_{n-1}f_{n-1}d_n^{P}(e_{n,i}),\quad n\geq 1.
\]
Then the maps \(f_*\) form a chain map \(P_*\to C_*\)
lifting \(f\).
\end{proposition}

\begin{lemma}[{\cite[Lemma~2.4]{IIVZ}}]\label{lem:phi}
Let \(A\) be local with radical \(J=\rad A\), and suppose that the differential
of the resolution \(P_*\) has the following form on basis elements:
\[
        d_n^{P}(e_{n,i})
        =\sum_{p\in T_{n,i}} a_p\,e_{n-1,j_p}\,b_p
        +\sum_{q\in T'_{n,i}} e_{n-1,j_q}\,b'_q,
        \qquad a_p\in J,\ b_p,b'_q\in A.
\]
If \(\Phi_*:P_*\to\Bbar_*(A)\) is the chain map constructed by
Proposition~\ref{prop:lift} with \(f=\id_A\) using the self-homotopy
\(s_*\), then
\[
        \Phi_n(e_{n,i})
        =\sum_{p\in T_{n,i}}
        s_{n-1}\bigl(a_p\,\Phi_{n-1}(e_{n-1,j_p})\,b_p\bigr).
\]
In particular, terms of the differential with purely right coefficients
do not contribute.
\end{lemma}

\begin{lemma}[{\cite[Lemma~2.5]{IIVZ}}]\label{lem:psi}
Let \(t_*\) be a weak self-homotopy of the resolution \(P_*\) consisting of
right \(A\)-linear maps and satisfying \(t_{n+1}t_n=0\). If
\(\Psi_*:\Bbar_*(A)\to P_*\) is the chain map constructed by
Proposition~\ref{prop:lift} with \(f=\id_A\) using \(t_*\), then for
\(n\geq 1\) and \(a_1,\dots,a_n\in A\)
\[
        \Psi_n(1\otimes\overline a_1\otimes\cdots\otimes\overline a_n
        \otimes 1)
        =t_{n-1}\bigl(a_1\,
        \Psi_{n-1}(1\otimes\overline a_2\otimes\cdots\otimes
        \overline a_n\otimes 1)\bigr).
\]
\end{lemma}

\section{Weak self-homotopy}\label{sec:homotopy}

Set
\[
        t_{-1}(r)=e_{000}r,
        \qquad r\in R.
\]
For \(n\geq 0\), the map \(t_n:Q_n\to Q_{n+1}\) is defined by right \(R\)-linearity. It therefore suffices to give the values \(t_n(\beta e_{ijk})\), where \(\beta\in B\) and \(i+j+k=n\). Below, \(\mathbf{1}_C\) denotes the indicator of the condition \(C\).

\[
\begin{aligned}
&t_n(e_{ijk})=0,\\
&t_n(x_0e_{ijk})=
\begin{cases}
\rho^{-j}e_{i+1,j,k},& k>0\ \text{or}\ j=0,\\
0,& k=0,\ j>0,
\end{cases}\\
&t_n(x_1e_{ijk})=
\begin{cases}
e_{i,j+1,0},& k=0,\\
0,& k>0,
\end{cases}\\
&t_n(x_2e_{ijk})=
\begin{cases}
e_{0,j,k+1},& i=0,\\
0,& i>0,
\end{cases}\\
&t_n(ae_{ijk})=
\begin{cases}
x_0e_{i,j+1,0}+\mathbf{1}_{j=0}\,\rho^i e_{i+1,0,0}x_1,& k=0,\\
0,& k>0,
\end{cases}\\
&t_n(be_{ijk})=\rho^{-i}x_1e_{i,j,k+1}+\mathbf{1}_{k=0}\,\rho^{j-i}e_{i,j+1,0}x_2,\\
&t_n(ce_{ijk})=\rho^{-j}x_2e_{i+1,j,k}+\mathbf{1}_{j>0}\,\rho^{k-i-j}x_1e_{i+1,j-1,k+1}+\mathbf{1}_{i=0}\,\rho^{k-j}e_{0,j,k+1}x_0+\mathbf{1}_{j>0,\ k=0}\,\rho^{-i-1}e_{i+1,j,0}x_2,\\
&t_n(we_{ijk})=\rho^{-i}ae_{i,j,k+1}+\mathbf{1}_{k=0}\,\rho^{j-i}x_0e_{i,j+1,0}x_2+\mathbf{1}_{j=0,\ k=0}\,e_{i+1,0,0}b.
\end{aligned}
\]

\begin{proposition}\label{prop:homotopy}
The maps $t_n$ defined above form a weak self-homotopy of the complex
\[
        Q_*\xrightarrow{\ \varepsilon\ }R\longrightarrow0 .
\]
Namely,
\[
        \varepsilon\,t_{-1}=\id_R,
\]
and, for all $n\ge0$,
\[
        d_{n+1}t_n+t_{n-1}d_n=\id_{Q_n}
        \qquad(d_0=\varepsilon).
\]
Moreover, the following normalization holds:
\[
        t_{n+1}t_n=0,
        \qquad n\ge-1 .
\]
\end{proposition}

\begin{proof}
Both sides are right $R$-linear, and the elements $\beta e_{ijk}$,
$\beta\in B$, $i+j+k=n$, form a basis of $Q_n$ as a right
$R$-module; the relations are verified directly on these elements.
\end{proof}

\section{Construction of comparison morphisms}\label{sec:comparison}

\subsection{The morphism \(\Phi_*:Q_*\to\Bbar_*(R)\)}

Throughout the rest of the paper, we set \(A=R\), \(P_*=Q_*\), \(f=\id_R\), and use
the weak self-homotopy \(t_*\) constructed in the preceding section and
the standard self-homotopy \(s_*\) of the normalized bar resolution. We
take \(B_A\) to be \(B\setminus\{1\}\).

For a word \(\varepsilon=(\varepsilon_1,\ldots,\varepsilon_n)\in\{0,1,2\}^n\),
set
\begin{equation}\label{eq:T-word}
        T(\varepsilon)
        =1\otimes\overline x_{\varepsilon_1}\otimes\cdots\otimes
        \overline x_{\varepsilon_n}\otimes 1
        \in\Bbar_n(R);
\end{equation}
the empty word corresponds to \(T(\;)=1\otimes 1\).

The algebra \(R\) is local: the ideal \(J=\rad R=\langle x_0,x_1,x_2\rangle\) has
\(K\)-basis \(B\setminus\{1\}\), is nilpotent (\(J^4=0\)), and
\(R/J\cong K\). The differential \(d_n(e_{ijk})\) has the form required
by Lemma~\ref{lem:phi}: the terms
\[
        \rho^{j}x_0e_{i-1,j,k},
        \qquad
        \rho^{k}x_1e_{i,j-1,k},
        \qquad
        \rho^{i}x_2e_{i,j,k-1}
\]
have left coefficients in \(J\), while the remaining terms have purely right coefficients
and, by Lemma~\ref{lem:phi}, do not contribute. Hence
\(\Phi_0(e_{000})=1\otimes 1\), and for \(n\geq 1\)
\begin{equation}\label{eq:Phi-recursion}
        \Phi_n(e_{ijk})
        =\mathbf 1_{i>0}\,\rho^{j}\,
        s_{n-1}\bigl(x_0\,\Phi_{n-1}(e_{i-1,j,k})\bigr)
        +\mathbf 1_{j>0}\,\rho^{k}\,
        s_{n-1}\bigl(x_1\,\Phi_{n-1}(e_{i,j-1,k})\bigr)
        +\mathbf 1_{k>0}\,\rho^{i}\,
        s_{n-1}\bigl(x_2\,\Phi_{n-1}(e_{i,j,k-1})\bigr).
\end{equation}
On the generators \(e_{ijk}\), \(i+j+k=n\), the morphism \(\Phi_n\) is defined
by~\eqref{eq:Phi-recursion} and extended to the whole module \(Q_n\) by
\(\Lambda\)-linearity.

\begin{lemma}\label{lem:Phi-support}
For all \(i,j,k\geq 0\), \(i+j+k=n\), the element \(\Phi_n(e_{ijk})\) lies in the
\(K\)-linear span of the tensors \(T(\varepsilon)\),
\(\varepsilon\in\{0,1,2\}^n\). Moreover, for every
\(\varepsilon_0\in\{0,1,2\}\) and every word
\(\varepsilon'\in\{0,1,2\}^{n-1}\)
\begin{equation}\label{eq:s-prepend}
        s_{n-1}\bigl(x_{\varepsilon_0}\,T(\varepsilon')\bigr)
        =T\bigl((\varepsilon_0,\varepsilon')\bigr).
\end{equation}
\end{lemma}

\begin{proof}
We first verify~\eqref{eq:s-prepend}. The outer left coefficient of the tensor
\(T(\varepsilon')\) is \(1\), so
\[
        x_{\varepsilon_0}\,T(\varepsilon')
        =x_{\varepsilon_0}\otimes\overline x_{\varepsilon'_1}
        \otimes\cdots\otimes\overline x_{\varepsilon'_{n-1}}\otimes 1,
\]
and the definition of the self-homotopy \(s_*\) gives
\[
        s_{n-1}\bigl(x_{\varepsilon_0}\,T(\varepsilon')\bigr)
        =1\otimes\overline x_{\varepsilon_0}\otimes
        \overline x_{\varepsilon'_1}\otimes\cdots\otimes
        \overline x_{\varepsilon'_{n-1}}\otimes 1
        =T\bigl((\varepsilon_0,\varepsilon')\bigr).
\]
The right-hand side is nonzero since \(x_{\varepsilon_0}\notin K\cdot 1\).

The first assertion follows by induction on \(n\). For \(n=0\), we have
\(\Phi_0(e_{000})=1\otimes 1=T(\;)\). Let \(n\geq 1\), and suppose that the assertion
holds for \(n-1\). Then each of the values \(\Phi_{n-1}(e_{i-1,j,k})\),
\(\Phi_{n-1}(e_{i,j-1,k})\), \(\Phi_{n-1}(e_{i,j,k-1})\) is a \(K\)-linear
combination of tensors \(T(\varepsilon')\), \(\varepsilon'\in\{0,1,2\}^{n-1}\).
Substituting into~\eqref{eq:Phi-recursion} and applying~\eqref{eq:s-prepend} to
each term gives a \(K\)-linear combination of tensors
\(T(\varepsilon)\), \(\varepsilon\in\{0,1,2\}^n\).
\end{proof}

\begin{proposition}\label{prop:words-paths}
Let $W(i,j,k)$ be the set of words
$\varepsilon\in\{0,1,2\}^n$ containing exactly
$i$ occurrences of $0$, $j$ occurrences of $1$, and $k$ occurrences of $2$,
where $i,j,k\ge0$ and $n=i+j+k$.
Then the correspondence
\begin{equation}\label{eq:word-to-path}
\varepsilon\longmapsto(v_0,\ldots,v_n),
\qquad
v_m=(i,j,k)-\sum_{\ell=1}^{m}u_{\varepsilon_\ell},
\quad 0\le m\le n,
\end{equation}
where $u_0,u_1,u_2$ are the standard basis vectors of
$\mathbb Z^3$, defines a bijection between $W(i,j,k)$
and the set of lattice paths in $\mathbb Z_{\ge0}^3$
from $(i,j,k)$ to $(0,0,0)$ with steps $-u_0,-u_1,-u_2$.
\end{proposition}

\begin{proof}
For $\varepsilon\in W(i,j,k)$,
formula~\eqref{eq:word-to-path} gives
\[
v_m=\sum_{\ell=m+1}^{n}u_{\varepsilon_\ell},
\qquad 0\le m\le n.
\]
Hence $v_m\in\mathbb Z_{\ge0}^3$,
$v_n=(0,0,0)$, and
$v_{m-1}-v_m=u_{\varepsilon_m}$ for $1\le m\le n$.

The inverse correspondence is given by the equalities
$u_{\varepsilon_m}=v_{m-1}-v_m$, $1\le m\le n$.
Since
$\sum_{m=1}^{n}u_{\varepsilon_m}=v_0-v_n=(i,j,k)$,
the recovered word belongs to $W(i,j,k)$.
\end{proof}

Define the weight of a word by
\begin{equation}\label{eq:weight}
        w(\varepsilon)=\prod_{m=1}^{n}\kappa_{\varepsilon_m}(v_{m-1}),
        \qquad
        \kappa_0(\alpha,\beta,\gamma)=\rho^{\beta},
        \quad
        \kappa_1(\alpha,\beta,\gamma)=\rho^{\gamma},
        \quad
        \kappa_2(\alpha,\beta,\gamma)=\rho^{\alpha};
\end{equation}
the factor \(\kappa_{\varepsilon_m}(v_{m-1})\) is the coefficient with which
the term corresponding to the \(m\)-th step enters
the recursion~\eqref{eq:Phi-recursion} at the vertex \(v_{m-1}\).

\begin{proposition}\label{prop:phi-paths}
For all \(i,j,k\geq 0\), \(i+j+k=n\),
\[
        \Phi_n(e_{ijk})
        =\sum_{\varepsilon\in W(i,j,k)}
        w(\varepsilon)\,T(\varepsilon).
\]
\end{proposition}

\begin{proof}
We argue by induction on \(n\). For \(n=0\), the set \(W(0,0,0)\) consists of the empty
word of weight \(1\), and both sides are equal to \(1\otimes 1\).

Let \(n\geq 1\), and suppose that the assertion holds for \(n-1\). Partition \(W(i,j,k)\) according to
the first letter:
\[
        W(i,j,k)
        =\bigl\{(0,\varepsilon')\mid\varepsilon'\in W(i-1,j,k)\bigr\}
        \sqcup\bigl\{(1,\varepsilon')\mid\varepsilon'\in W(i,j-1,k)\bigr\}
        \sqcup\bigl\{(2,\varepsilon')\mid\varepsilon'\in W(i,j,k-1)\bigr\},
\]
where a term is absent if the corresponding coordinate is zero. For
a word \(\varepsilon=(\varepsilon_0,\varepsilon')\), the first step of the corresponding
path starts at the vertex \(v_0=(i,j,k)\), and the remaining steps form
the path corresponding to \(\varepsilon'\); hence, by the definition
of the weight~\eqref{eq:weight},
\[
        w\bigl((\varepsilon_0,\varepsilon')\bigr)
        =\kappa_{\varepsilon_0}(i,j,k)\,w(\varepsilon'),
        \qquad
        \kappa_0(i,j,k)=\rho^{j},
        \quad
        \kappa_1(i,j,k)=\rho^{k},
        \quad
        \kappa_2(i,j,k)=\rho^{i}.
\]
Substitute the induction hypothesis into~\eqref{eq:Phi-recursion} and apply
Lemma~\ref{lem:Phi-support}; for the first term, this gives
\[
        \rho^{j}\,s_{n-1}\bigl(x_0\,\Phi_{n-1}(e_{i-1,j,k})\bigr)
        =\rho^{j}\sum_{\varepsilon'\in W(i-1,j,k)}
        w(\varepsilon')\,s_{n-1}\bigl(x_0\,T(\varepsilon')\bigr)
        =\sum_{\varepsilon'\in W(i-1,j,k)}
        w\bigl((0,\varepsilon')\bigr)\,T\bigl((0,\varepsilon')\bigr),
\]
and similarly for the terms with \(x_1\) and \(x_2\), with weights of the first step
\(\rho^{k}\) and \(\rho^{i}\). By the above partition, summing the three resulting expressions
gives the required formula.
\end{proof}

\subsection{The morphism \(\Psi_*:\Bbar_*(R)\to Q_*\)}

By Proposition~\ref{prop:homotopy}, the maps \(t_*\) are right
\(R\)-linear and satisfy \(t_{n+1}t_n=0\), so
Lemma~\ref{lem:psi} applies. It gives \(\Psi_0(1\otimes 1)=e_{000}\) and, for
\(n\geq 1\) and \(b_1,\ldots,b_n\in B\setminus\{1\}\),
\begin{equation}\label{eq:Psi-recursion}
        \Psi_n\bigl(1\otimes\overline b_1\otimes\cdots\otimes
        \overline b_n\otimes 1\bigr)
        =t_{n-1}\bigl(b_1\,
        \Psi_{n-1}\bigl(1\otimes\overline b_2\otimes\cdots\otimes
        \overline b_n\otimes 1\bigr)\bigr).
\end{equation}
The elements \(1\otimes\overline b_1\otimes\cdots\otimes\overline b_n\otimes 1\)
form a basis of \(\Bbar_n(R)\) as a free \(\Lambda\)-module, so
formula~\eqref{eq:Psi-recursion} uniquely determines \(\Psi_n\); the map is
extended to the whole module \(\Bbar_n(R)\) by \(\Lambda\)-linearity.

\begin{proposition}\label{prop:psi1}
In degree \(1\),
\[
\begin{aligned}
        \Psi_1(1\otimes\overline x_0\otimes 1)&=e_{100},
        &
        \Psi_1(1\otimes\overline x_1\otimes 1)&=e_{010},
        &
        \Psi_1(1\otimes\overline x_2\otimes 1)&=e_{001},
        \\[1mm]
        \Psi_1(1\otimes\overline a\otimes 1)&=x_0e_{010}+e_{100}x_1,
        &
        \Psi_1(1\otimes\overline b\otimes 1)&=x_1e_{001}+e_{010}x_2,
        &
        \Psi_1(1\otimes\overline c\otimes 1)&=x_2e_{100}+e_{001}x_0,
\end{aligned}
\]
\[
        \Psi_1(1\otimes\overline w\otimes 1)
        =ae_{001}+x_0e_{010}x_2+e_{100}b.
\]
\end{proposition}

\begin{proof}
By formula~\eqref{eq:Psi-recursion},
\(\Psi_1(1\otimes\overline\beta\otimes 1)=t_0(\beta e_{000})\) for
\(\beta\in B\setminus\{1\}\); it remains to substitute \((i,j,k)=(0,0,0)\) into
the defining formulas for \(t\).
\end{proof}

\medskip
\noindent\textbf{Distinguished words.}
We use power notation for words: \(0^i\) denotes \(i\) consecutive occurrences of the symbol \(0\), and similarly for \(1^j\), \(2^k\); blocks with exponent
zero are omitted. For \(i,j,k\geq 0\), \(i+j+k=n\), we distinguish the
following word in the set \(W(i,j,k)\):
\begin{equation}\label{eq:distinguished-word}
        \varepsilon^{*}_{ijk}
        =
        \begin{cases}
        (1^j,0^i),& k=0,\\[1mm]
        (0^i,2^k,1^j),& k>0.
\end{cases}
\end{equation}
The order of the blocks for \(k=0\) differs from that obtained by formally
substituting \(k=0\) into the second line of~\eqref{eq:distinguished-word}; this
corresponds to the case distinction in~\eqref{eq:t-x0} below. Since the word
\(\varepsilon^{*}_{ijk}\) contains exactly \(i\) occurrences of \(0\), \(j\)
occurrences of \(1\), and \(k\) occurrences of \(2\), the assignment
\((i,j,k)\mapsto\varepsilon^{*}_{ijk}\) is injective. A word that is not
distinguished for any triple \((i,j,k)\) is called non-distinguished.

\begin{proposition}\label{prop:sparse-Psi}
Let \(i,j,k\geq 0\) and \(i+j+k=n\). Then
\begin{enumerate}[label=\textup{(\roman*)}]
\item for the distinguished word
\[
        \Psi_n\bigl(T(\varepsilon^{*}_{ijk})\bigr)
        =
        \begin{cases}
        e_{ij0},& k=0,\\[1mm]
        \rho^{-ij}e_{ijk},& k>0;
        \end{cases}
\]
\item \(\Psi_n\bigl(T(\varepsilon)\bigr)=0\) for every non-distinguished
word \(\varepsilon\in\{0,1,2\}^n\).
\end{enumerate}
In particular, among the words in \(W(i,j,k)\), exactly one word has nonzero
image under \(\Psi_n\), namely \(\varepsilon^{*}_{ijk}\).
\end{proposition}

\begin{proof}
We use the following specializations of the weak self-homotopy:
\begin{equation}\label{eq:t-x0}
        t_m(x_0e_{IJK})
        =
        \begin{cases}
        \rho^{-J}e_{I+1,J,K},& K>0\ \text{or}\ J=0,\\
        0,& K=0,\ J>0,
        \end{cases}
\end{equation}
\begin{equation}\label{eq:t-x1}
        t_m(x_1e_{IJK})
        =
        \begin{cases}
        e_{I,J+1,0},& K=0,\\
        0,& K>0,
        \end{cases}
        \qquad\qquad
        t_m(x_2e_{IJK})
        =
        \begin{cases}
        e_{0,J,K+1},& I=0,\\
        0,& I>0.
        \end{cases}
\end{equation}
We prove the two assertions simultaneously by induction on \(n\). Let \(A_n\)
denote assertion \textup{(i)} for all triples \(i+j+k=n\), and let
\(B_n\) denote assertion \textup{(ii)} for all words of length \(n\).

\medskip
\noindent\textbf{Base case.}
For \(n=0\), there is a unique word, the empty word; it is distinguished and corresponds
to the triple \((i,j,k)=(0,0,0)\). By definition,
\(\Psi_0(1\otimes 1)=e_{000}\), so \(A_0\) holds. There are no non-distinguished
words of length \(0\), so \(B_0\) also holds.

\medskip
\noindent\textbf{Induction hypothesis.}
Let \(n\geq 1\), and suppose that \(A_{n-1}\) and \(B_{n-1}\) have been proved.
Fix a word
\[
        \varepsilon=(r,\varepsilon_2,\ldots,\varepsilon_n),
        \qquad r\in\{0,1,2\},
        \qquad
        \varepsilon'=(\varepsilon_2,\ldots,\varepsilon_n).
\]
By formula~\eqref{eq:Psi-recursion},
\begin{equation}\label{eq:Psi-step}
        \Psi_n\bigl(T(\varepsilon)\bigr)
        =t_{n-1}\bigl(x_r\,\Psi_{n-1}\bigl(T(\varepsilon')\bigr)\bigr).
\end{equation}

\medskip
\noindent\textbf{Suffixes of distinguished words.}
The suffix of every nonempty distinguished word is again a distinguished
word. Indeed, let \(k=0\), that is,
\(\varepsilon^{*}_{ij0}=(1^j,0^i)\). For \(j>0\), the suffix is
\[
        (1^{j-1},0^i)=\varepsilon^{*}_{i,j-1,0},
\]
and for \(j=0\), \(i>0\), it is
\[
        (0^{i-1})=\varepsilon^{*}_{i-1,0,0}.
\]
Now let \(k>0\), that is, \(\varepsilon^{*}_{ijk}=(0^i,2^k,1^j)\).
For \(i>0\), the suffix is
\[
        (0^{i-1},2^k,1^j)=\varepsilon^{*}_{i-1,j,k};
\]
for \(i=0\), \(k>1\), it is
\[
        (2^{k-1},1^j)=\varepsilon^{*}_{0,j,k-1};
\]
for \(i=0\), \(k=1\), it is
\[
        (1^j)=\varepsilon^{*}_{0,j,0}.
\]
In particular, by contraposition,
\begin{equation}\label{eq:contrapositive}
        \varepsilon'\ \text{is non-distinguished}
        \quad\Longrightarrow\quad
        \varepsilon\ \text{is non-distinguished}.
\end{equation}

\medskip
\noindent\textbf{Case 1: the suffix \(\varepsilon'\) is non-distinguished.}
By the induction hypothesis \(B_{n-1}\),
\[
        \Psi_{n-1}\bigl(T(\varepsilon')\bigr)=0,
\]
and hence~\eqref{eq:Psi-step} gives
\[
        \Psi_n\bigl(T(\varepsilon)\bigr)=t_{n-1}(x_r\cdot 0)=0.
\]
By~\eqref{eq:contrapositive}, the word \(\varepsilon\) is also non-distinguished, in
agreement with \(B_n\). No distinguished word of length \(n\) falls into this
case, since its suffix is distinguished.

\medskip
\noindent\textbf{Case 2: the suffix is distinguished,}
\(\varepsilon'=\varepsilon^{*}_{IJK}\), \(I+J+K=n-1\).
By the induction hypothesis \(A_{n-1}\),
\begin{equation}\label{eq:Psi-suffix}
        \Psi_{n-1}\bigl(T(\varepsilon')\bigr)
        =
        \begin{cases}
        e_{IJ0},& K=0,\\[1mm]
        \rho^{-IJ}e_{IJK},& K>0.
        \end{cases}
\end{equation}

\smallskip
\noindent\textbf{Subcase 2.1: \(K=0\),}
\(\varepsilon'=(1^J,0^I)\), \(\Psi_{n-1}(T(\varepsilon'))=e_{IJ0}\).

\smallskip
\noindent\emph{Subcase 2.1.1: \(r=1\).}
It follows from~\eqref{eq:t-x1} that
\[
        \Psi_n\bigl(T(\varepsilon)\bigr)
        =t_{n-1}(x_1e_{IJ0})
        =e_{I,J+1,0},
        \qquad
        \varepsilon=(1,1^J,0^I)=(1^{J+1},0^I)
        =\varepsilon^{*}_{I,J+1,0}.
\]

\smallskip
\noindent\emph{Subcase 2.1.2: \(r=2\).}
It follows from~\eqref{eq:t-x1} that
\[
        t_{n-1}(x_2e_{IJ0})
        =
        \begin{cases}
        e_{0,J,1},& I=0,\\
        0,& I>0.
        \end{cases}
\]
If \(I=0\), then
\[
        \Psi_n\bigl(T(\varepsilon)\bigr)=e_{0,J,1}=\rho^{-0\cdot J}e_{0,J,1},
        \qquad
        \varepsilon=(2,1^J)=\varepsilon^{*}_{0,J,1}.
\]
If \(I>0\), then
\[
        \Psi_n\bigl(T(\varepsilon)\bigr)=0,
        \qquad
        \varepsilon=(2,1^J,0^I);
\]
the word \(\varepsilon\) is non-distinguished, since in a distinguished word with
a positive number of symbols \(2\), all symbols \(0\) occur to the left of
all symbols \(2\).

\smallskip
\noindent\emph{Subcase 2.1.3: \(r=0\).}
It follows from~\eqref{eq:t-x0} that
\[
        t_{n-1}(x_0e_{IJ0})
        =
        \begin{cases}
        e_{I+1,0,0},& J=0,\\
        0,& J>0.
        \end{cases}
\]
If \(J=0\), then
\[
        \Psi_n\bigl(T(\varepsilon)\bigr)=e_{I+1,0,0},
        \qquad
        \varepsilon=(0,0^I)=(0^{I+1})=\varepsilon^{*}_{I+1,0,0}.
\]
If \(J>0\), then
\[
        \Psi_n\bigl(T(\varepsilon)\bigr)=0,
        \qquad
        \varepsilon=(0,1^J,0^I);
\]
the word \(\varepsilon\) is non-distinguished, since in the absence of symbols \(2\),
a distinguished word has the form \((1^j,0^i)\).

\smallskip
\noindent\textbf{Subcase 2.2: \(K>0\),}
\(\varepsilon'=(0^I,2^K,1^J)\),
\(\Psi_{n-1}(T(\varepsilon'))=\rho^{-IJ}e_{IJK}\).

\smallskip
\noindent\emph{Subcase 2.2.1: \(r=0\).}
Since \(K>0\), it follows from~\eqref{eq:t-x0} that
\(t_{n-1}(x_0e_{IJK})=\rho^{-J}e_{I+1,J,K}\), and hence
\[
        \Psi_n\bigl(T(\varepsilon)\bigr)
        =\rho^{-IJ}\,t_{n-1}(x_0e_{IJK})
        =\rho^{-IJ}\rho^{-J}e_{I+1,J,K}
        =\rho^{-(I+1)J}e_{I+1,J,K},
\]
and in this case
\[
        \varepsilon=(0,0^I,2^K,1^J)=(0^{I+1},2^K,1^J)
        =\varepsilon^{*}_{I+1,J,K}.
\]

\smallskip
\noindent\emph{Subcase 2.2.2: \(r=1\).}
Since \(K>0\), it follows from~\eqref{eq:t-x1} that
\[
        t_{n-1}(x_1e_{IJK})=0,
        \qquad
        \Psi_n\bigl(T(\varepsilon)\bigr)=0,
        \qquad
        \varepsilon=(1,0^I,2^K,1^J);
\]
the word \(\varepsilon\) is non-distinguished, since when the number of symbols
\(2\) is positive, all symbols \(1\) in a distinguished word occur to the right of all
symbols \(2\).

\smallskip
\noindent\emph{Subcase 2.2.3: \(r=2\).}
It follows from~\eqref{eq:t-x1} that
\[
        t_{n-1}(x_2e_{IJK})
        =
        \begin{cases}
        e_{0,J,K+1},& I=0,\\
        0,& I>0.
        \end{cases}
\]
If \(I=0\), then the coefficient \(\rho^{-IJ}\) is \(1\) and
\[
        \Psi_n\bigl(T(\varepsilon)\bigr)=e_{0,J,K+1},
        \qquad
        \varepsilon=(2,2^K,1^J)=(2^{K+1},1^J)
        =\varepsilon^{*}_{0,J,K+1}.
\]
If \(I>0\), then
\[
        \Psi_n\bigl(T(\varepsilon)\bigr)=0,
        \qquad
        \varepsilon=(2,0^I,2^K,1^J);
\]
the word \(\varepsilon\) is non-distinguished, since in a distinguished word with
a positive number of symbols \(2\), all symbols \(0\) occur to the left of
all symbols \(2\).

\medskip
\noindent\textbf{Exhaustiveness of the cases.}
Every nonempty distinguished word decomposes uniquely into a pair consisting of its first
letter and its distinguished suffix:
\begin{itemize}[leftmargin=2em]
\item \(\varepsilon^{*}_{ij0}=(1,\varepsilon^{*}_{i,j-1,0})\) for
\(j>0\) --- Subcase~2.1.1;
\item \(\varepsilon^{*}_{i00}=(0,\varepsilon^{*}_{i-1,0,0})\) for
\(i>0\) --- Subcase~2.1.3 with \(J=0\);
\item \(\varepsilon^{*}_{ijk}=(0,\varepsilon^{*}_{i-1,j,k})\) for
\(k>0\), \(i>0\) --- Subcase~2.2.1;
\item \(\varepsilon^{*}_{0jk}=(2,\varepsilon^{*}_{0,j,k-1})\) for
\(k>1\) --- Subcase~2.2.3 with \(I=0\);
\item \(\varepsilon^{*}_{0j1}=(2,\varepsilon^{*}_{0,j,0})\) ---
Subcase~2.1.2 with \(I=0\).
\end{itemize}
In all the cases listed above, the computed values agree with the formulas in
assertion \(A_n\), which is therefore proved. Now let
\(\varepsilon\) be a non-distinguished word of length \(n\). If its suffix
is non-distinguished, then \(\Psi_n(T(\varepsilon))=0\) by Case 1; if its suffix
is distinguished, then the word is covered by one of the six subcases of Case 2 and,
being non-distinguished, belongs to one of the cases in which the image is zero. This proves
\(B_n\), and the proposition holds for all \(n\geq 0\) by induction.
\end{proof}

\begin{corollary}\label{cor:Psi-Phi-identity}
The comparison morphisms constructed above satisfy
\[
        \Psi_*\Phi_*=\id_{Q_*}.
\]
\end{corollary}

\begin{proof}
Both morphisms are \(\Lambda\)-linear, and the elements \(e_{ijk}\), \(i+j+k=n\),
form a \(\Lambda\)-basis of \(Q_n\); hence it suffices to verify
the equality \(\Psi_n\Phi_n(e_{ijk})=e_{ijk}\).

By Proposition~\ref{prop:phi-paths} and
Proposition~\ref{prop:sparse-Psi}, in the sum
\[
        \Psi_n\Phi_n(e_{ijk})
        =\sum_{\varepsilon\in W(i,j,k)}
        w(\varepsilon)\,\Psi_n\bigl(T(\varepsilon)\bigr)
\]
only the summand corresponding to \(\varepsilon^{*}_{ijk}\) is nonzero, so
that
\begin{equation}\label{eq:Psi-Phi-distinguished}
        \Psi_n\Phi_n(e_{ijk})
        =w(\varepsilon^{*}_{ijk})\,
        \Psi_n\bigl(T(\varepsilon^{*}_{ijk})\bigr).
\end{equation}

Let \(k=0\). The word \(\varepsilon^{*}_{ij0}=(1^j,0^i)\) corresponds to the path
\[
        (i,j,0)\xrightarrow{\ 1^j\ }(i,0,0)\xrightarrow{\ 0^i\ }(0,0,0).
\]
The first \(j\) steps start at vertices with third coordinate \(0\), and the last
\(i\) steps start at vertices with second coordinate \(0\). Thus,
by~\eqref{eq:weight}, each step has weight \(\rho^{0}=1\), and
\(w(\varepsilon^{*}_{ij0})=1\). Together with
Proposition~\ref{prop:sparse-Psi}\textup{(i)},
formula~\eqref{eq:Psi-Phi-distinguished} gives
\(\Psi_n\Phi_n(e_{ij0})=e_{ij0}\).

Let \(k>0\). The word \(\varepsilon^{*}_{ijk}=(0^i,2^k,1^j)\) corresponds to the path
\[
        (i,j,k)\xrightarrow{\ 0^i\ }(0,j,k)
        \xrightarrow{\ 2^k\ }(0,j,0)
        \xrightarrow{\ 1^j\ }(0,0,0).
\]
The first \(i\) steps start at vertices with second coordinate \(j\) and each have
weight \(\rho^{j}\); the next \(k\) steps start at vertices with first
coordinate \(0\), and the last \(j\) steps start at vertices with third
coordinate \(0\), and their weights are \(1\). Hence
\(w(\varepsilon^{*}_{ijk})=\rho^{ij}\), and
formula~\eqref{eq:Psi-Phi-distinguished}, together with
Proposition~\ref{prop:sparse-Psi}\textup{(i)}, gives
\[
        \Psi_n\Phi_n(e_{ijk})
        =\rho^{ij}\rho^{-ij}e_{ijk}
        =e_{ijk}. \qedhere
\]
\end{proof}

\parttitle{Part II. The BV structure on \(\HH^*(R)\)}

\section{Complete computation of the BV operator}\label{sec:BV}

In the preceding sections, we constructed the comparison morphisms
\[
        \Psi_*:\Bbar_*(R)\longrightarrow Q_*,
        \qquad
        \Phi_*:Q_*\longrightarrow\Bbar_*(R),
        \qquad
        \Psi_*\Phi_*=\id_{Q_*}.
\]
In this section, the operator $\Delta$ is transferred to the minimal
resolution and computed on products of the multiplicative
generators of the algebra $\HH^*(R)$.

\subsection{The symmetric structure of the algebra}

Define a linear functional
\[
        \tau:R\longrightarrow K
\]
as the coefficient of
\[
        w=x_0x_1x_2
\]
in the expansion of an element in the basis
\[
        B=\{1,x_0,x_1,x_2,a,b,c,w\}.
\]
Set
\[
        \langle u,v\rangle=\tau(uv).
\]

\begin{proposition}\label{prop:BV-symmetric-revised}
The form \(\langle-,-\rangle\) is associative, symmetric, and nondegenerate.
Consequently, \(R\) is a symmetric algebra.
\end{proposition}

\begin{proof}
Associativity follows from the equality
\[
        \langle uv,z\rangle=\tau(uvz)=\langle u,vz\rangle.
\]
The nonzero pairings of basis elements are
\[
        \langle1,w\rangle=\langle w,1\rangle=1,
\]
\[
        \langle x_0,b\rangle=\langle b,x_0\rangle=1,
        \qquad
        \langle x_1,c\rangle=\langle c,x_1\rangle=1,
\]
\[
        \langle x_2,a\rangle=\langle a,x_2\rangle=1.
\]
For example,
\[
        x_0b=w,
\]
and also
\[
        bx_0=x_1x_2x_0
        =\rho^{-1}x_1x_0x_2
        =x_0x_1x_2=w.
\]
The remaining equalities follow by cyclically permuting the indices.
Thus, the dual basis is
\[
\begin{array}{c|cccccccc}
 u&1&x_0&x_1&x_2&a&b&c&w\\ \hline
 u^\vee&w&b&c&a&x_2&x_0&x_1&1.
\end{array}
\]
This proves both symmetry and nondegeneracy of the form.
\end{proof}

\subsection{Transferring the operator $\Delta$ to the minimal resolution}

By Proposition~\ref{prop:BV-symmetric-revised}, the algebra $R$
is symmetric, so Theorem~\ref{thm:Tradler} applies: the algebra
$\HH^*(R)$ is a BV algebra, and on the normalized complex
the operator $\Delta$ is given by
\begin{equation}\label{eq:Tradler-R}
        \bigl\langle\Delta(f)(u_1,\dots,u_{n-1}),\,u_n\bigr\rangle
        =\sum_{r=1}^{n}
        \bigl\langle f(u_r,\dots,u_n,u_1,\dots,u_{r-1}),\,1\bigr\rangle,
        \qquad u_1,\dots,u_n\in R,
\end{equation}
where the signs disappear because $\Char K=2$.

The operator $\Delta$ is transferred to the minimal resolution by the
comparison morphisms: for $f\in\operatorname{Hom}_\Lambda(Q_n,R)$,
we set
\begin{equation}\label{eq:Delta-transfer-revised}
        \Delta(f)=\Delta(f\circ\Psi_n)\circ\Phi_{n-1}
        \in\operatorname{Hom}_\Lambda(Q_{n-1},R),
\end{equation}
and $\Delta=0$ in degree $0$. Since $\Phi_*$ and $\Psi_*$ form a
homotopy equivalence of resolutions, the induced operator on
$\HH^*(R)$ is independent of the choice of representing cocycles in
$\operatorname{Hom}_\Lambda(Q_*,R)$ and
$\operatorname{Hom}_\Lambda(\Bbar_*(R),R)$.

By nondegeneracy of the form, every element $z\in R$ admits the
expansion
\begin{equation}\label{eq:dual-expansion-revised}
        z=\sum_{v\in B}\langle z,v\rangle\,v^\vee,
\end{equation}
where $v^\vee$ is the corresponding dual basis element from
Proposition~\ref{prop:BV-symmetric-revised}.

\subsection{The cochain differential and the ring presentation}
\label{subsec:ring-presentation}

The complex computing $\HH^*(R)$ is obtained by applying
the functor $\operatorname{Hom}_\Lambda(-,R)$ to the minimal
resolution $Q_*$. For $r\in R$, write
\[
[r]_{\alpha,\beta,\gamma}
\in\operatorname{Hom}_\Lambda(Q_{\alpha+\beta+\gamma},R)
\]
for the cochain sending the free generator
$e_{\alpha,\beta,\gamma}$ to $r$ and all other
generators of the same degree to zero. The cochain differential
$\delta_Q(f)=f\circ d$ is given by the formula
\begin{equation}\label{eq:cochain-differential-revised}
\delta_Q([r]_{\alpha,\beta,\gamma})
=[\rho^\beta x_0r+\rho^\gamma r x_0]_{\alpha+1,\beta,\gamma}
+[\rho^\gamma x_1r+\rho^\alpha r x_1]_{\alpha,\beta+1,\gamma}
+[\rho^\alpha x_2r+\rho^\beta r x_2]_{\alpha,\beta,\gamma+1},
\end{equation}
which follows directly from the formulas for \(d\).

Let $\rho\ne1$ and $s=\operatorname{ord}(\rho)$.
We use the additive basis from
\cite[Propositions~3.1, 3.2]{GK06} and the systems of
multiplicative generators from
\cite[Theorems~5.12, 5.13]{GK06}.

The ring presentations in parts~II.2 and~II.3 of
Theorem~2.1 of that paper require a correction,
which will be established below.
We take the following cocycles as representatives of the multiplicative
generators:
\[
h=[w]_{0,0,0},
\]
\[
h_0=[b]_{0,0,0},
\qquad h_1=[c]_{0,0,0},
\qquad h_2=[a]_{0,0,0},
\]
\[
y_0=[x_0]_{1,0,0},
\qquad y_1=[x_1]_{0,1,0},
\qquad y_2=[x_2]_{0,0,1},
\]
\[
p=[1]_{1,1,1},
\]
\[
z_0^{(m)}=[b]_{m,0,0},
\qquad
z_1^{(m)}=[c]_{0,m,0},
\qquad
z_2^{(m)}=[a]_{0,0,m},
\]
where $m\ge1$ for $s=\infty$ and $1\le m<s$
for $1<s<\infty$. In the latter case, we also include
\[
p_0=[1]_{s,0,0},
\qquad p_1=[1]_{0,s,0},
\qquad p_2=[1]_{0,0,s},
\]
\[
t_0=[x_0]_{0,s-1,s-1},
\qquad
t_1=[x_1]_{s-1,0,s-1},
\qquad
t_2=[x_2]_{s-1,s-1,0}.
\]

Henceforth, cocycles and their cohomology classes are denoted
by the same letters.

\begin{remark}\label{rem:GK-presentations}
For the algebras $A_j=K[X_j]/I_j$, $j=7,8$, defined
in \cite[(2.22)--(2.27)]{GK06}, we have
\[
\begin{aligned}
A_7&\not\cong\HH^*(R),&&1<s<\infty,\\
A_8&\not\cong\HH^*(R),&&s=\infty,
\end{aligned}
\]
as graded $K$-algebras.

Indeed, by Proposition~3.2, formulas~(3.26),
(3.27), (3.30), and Theorem~4.2 of~\cite{GK06},
\[
\HH^1(R)
=\bigl\langle y_i,z_i^{(1)},h_i y_i
\mid i=0,1,2\bigr\rangle_K.
\]
The same product formula in $\HH^*(R)$ gives
\[
y_i^2=(z_i^{(1)})^2=(h_i y_i)^2=0,
\qquad i=0,1,2.
\]
Since $\HH^*(R)$ is commutative in characteristic~$2$,
we obtain
\[
u^2=0\qquad\text{for every }u\in\HH^1(R).
\]

On the other hand, consider the homomorphism
\[
\varphi:K[X_j]\longrightarrow K[T],\qquad
y_0\longmapsto T,\qquad
\text{all other generators}\longmapsto0.
\]
Under this substitution, all defining relations
$I_j$ vanish. Hence
$I_j\subseteq\ker\varphi$, and the substitution induces
a homomorphism $A_j\to K[T]$.
Since $y_0^2$ maps to $T^2\ne0$,
we have $y_0^2\ne0$ in $A_j$.

If there were an isomorphism of graded
$K$-algebras $\Phi:A_j\xrightarrow{\sim}\HH^*(R)$,
then $\Phi(y_0)\in\HH^1(R)$ would imply
\[
\Phi(y_0^2)=\Phi(y_0)^2=0,
\]
contradicting the injectivity of $\Phi$.
\end{remark}

\begin{theorem}\label{thm:ring-presentation}
There are isomorphisms of graded $K$-algebras
\[
\begin{aligned}
\HH^*(R)&\cong A_7/(y_0^2,y_1^2,y_2^2),
&&1<s<\infty,\\
\HH^*(R)&\cong A_8/(y_0^2,y_1^2,y_2^2),
&&s=\infty.
\end{aligned}
\]
\end{theorem}

\begin{proof}
The classes of the cocycles listed above generate $\HH^*(R)$
and, by the product formula in Theorem~4.2
of~\cite{GK06}, satisfy the relations $I_j$.
Consequently, they define an epimorphism of graded
$K$-algebras
\[
\pi_j:A_j\twoheadrightarrow\HH^*(R),
\]
where $j=7$ for $1<s<\infty$ and $j=8$ for $s=\infty$.

Set $J=(y_0^2,y_1^2,y_2^2)\subseteq A_j$.
The equalities $y_i^2=0$ in $\HH^*(R)$ imply
\[
J\subseteq\ker\pi_j.
\]

To prove the reverse inclusion, we use
the following families of monomials.

\emph{Case $1<s<\infty$.}
The relations \cite[(2.22)--(2.25)]{GK06},
together with the equalities $y_i^2=0$ modulo $J$,
reduce every monomial to zero or to a scalar multiple
of one of the following monomials:
\[
\begin{gathered}
p^r p_0^{a_0}p_1^{a_1}p_2^{a_2}
y_0^{\varepsilon_0}y_1^{\varepsilon_1}y_2^{\varepsilon_2},
\qquad 0\le r<s,\\
t_i p_{i+1}^{a}p_{i+2}^{b}
y_{i+1}^{\varepsilon}y_{i+2}^{\eta},\\
h_i p_i^a y_i^\varepsilon,\qquad
z_i^{(m)}p_i^a y_i^\varepsilon,
\qquad 1\le m<s,\qquad h.
\end{gathered}
\]
Here $i=0,1,2$, and
$\varepsilon,\eta,\varepsilon_i\in\{0,1\}$.

\emph{Case $s=\infty$.}
The relations \cite[(2.26), (2.27)]{GK06},
together with the equalities $y_i^2=0$ modulo $J$,
reduce every monomial to zero or to a scalar multiple
of one of the following monomials:
\[
\begin{gathered}
p^r y_0^{\varepsilon_0}y_1^{\varepsilon_1}y_2^{\varepsilon_2},
\qquad h_i y_i^\varepsilon,\\
z_i^{(m)}y_i^\varepsilon,\qquad h,
\qquad r\ge0,\quad m\ge1,
\end{gathered}
\]
where $i=0,1,2$, and $\varepsilon,\varepsilon_i\in\{0,1\}$.

Now let $a\in\ker\pi_j$. Modulo $J$, we have
\[
a-\sum_{\nu=1}^{N}\lambda_\nu M_\nu\in J,
\qquad \lambda_\nu\in K,
\]
where $M_\nu$ are pairwise distinct monomials
from the corresponding family listed above.

By Theorem~4.2 and Propositions~3.1, 3.2
of~\cite{GK06}, the images $\pi_j(M_\nu)$
are linearly independent.
Since $\pi_j(a)=0$ and $\pi_j(J)=0$, we obtain
\[
0=\pi_j(a)
=\sum_{\nu=1}^{N}\lambda_\nu\pi_j(M_\nu).
\]
It follows that $\lambda_1=\cdots=\lambda_N=0$,
so $a\in J$. Hence
$\ker\pi_j\subseteq J$, and finally
\[
\ker\pi_j=(y_0^2,y_1^2,y_2^2).
\]
The required isomorphisms follow from the first
isomorphism theorem.
\end{proof}

We retain the notation for the chosen cohomology classes
throughout. All products are taken in $\HH^*(R)$
with the presentation of Theorem~\ref{thm:ring-presentation}.

\subsection{Multigrading}

Assign to the generators of the free algebra
$T=K\langle x_0,x_1,x_2\rangle$ the vectors
\[
\deg_x(x_0)=(1,0,0),\qquad
\deg_x(x_1)=(0,1,0),\qquad
\deg_x(x_2)=(0,0,1),
\]
and extend $\deg_x$ additively to words:
$\deg_x(uv)=\deg_x(u)+\deg_x(v)$.
Let $T_d$ denote the $K$-linear span of the words
of multidegree $d=(d_0,d_1,d_2)\in\mathbb N^3$;
then $T=\bigoplus_{d\in\mathbb N^3}T_d$.

The generators of the ideal $I$, namely the elements $x_i^2$ and
$x_{i+1}x_i+\rho x_ix_{i+1}$, each lie
in a single component $T_d$: the former has multidegree
$2e_i$ and the latter $e_i+e_{i+1}$.
Hence $I=\bigoplus_d(I\cap T_d)$,
and $R=T/I$ inherits an $\mathbb N^3$-grading
\[
R=\bigoplus_{d\in\mathbb N^3}R_d,
\qquad
R_d=\frac{T_d+I}{I}.
\]

\begin{definition}\label{def:multihomogeneous}
An element $r\in R$ is called \emph{multihomogeneous}
of multidegree $d$ if $r\in R_d$; in this case,
we write $\deg_x(r)=d$.
Equivalently, all monomials in the expansion of $r$ in the basis $B$
contain exactly $d_0$ letters $x_0$, $d_1$ letters $x_1$,
and $d_2$ letters $x_2$.
The zero element belongs to every component $R_d$.
\end{definition}

For the basis \(B\),
\[
\begin{array}{c|cccccccc}
        r&1&x_0&x_1&x_2&a&b&c&w\\ \hline
        \deg_x(r)&(0,0,0)&(1,0,0)&(0,1,0)&(0,0,1)
        &(1,1,0)&(0,1,1)&(1,0,1)&(1,1,1)
\end{array}
\]

\subsection{Euler derivations and the table of \(D_i\)}

The computation of the action of \(\Delta\) begins with the degree-one classes
\(y_0,y_1,y_2\), for which, as will be shown below,
\(\Delta(y_i)=1\). The resulting operators
\[
        D_i:=[y_i,-]
\]
admit a closed-form description in terms of the Euler derivations of the algebra
\(R\), which allows us to compute the table of \(D_i\) below.

\begin{definition}\label{def:euler-derivation}
For \(i=0,1,2\), the \emph{Euler derivation} \(\partial_i:R\to R\)
is defined by the conditions
\[
        \partial_i(x_i)=x_i,
        \qquad
        \partial_i(x_j)=0\quad(j\ne i),
\]
and the Leibniz rule. The defining relations of \(R\) are multihomogeneous,
so \(\partial_i\) is well defined. Let
\(\deg_{x_i}(r)\) denote the \(i\)-th coordinate of the multidegree \(\deg_x(r)\); then,
on a multihomogeneous element \(r\in R\),
\[
        \partial_i(r)=\deg_{x_i}(r)\,r.
\]
\end{definition}

\begin{lemma}\label{lem:Euler-revised}
The bar cochain \(y_i\circ\Psi_1\) coincides with the Euler
derivation \(\partial_i\). Let
\[
F=[r]_{\alpha,\beta,\gamma}
\in\operatorname{Hom}_\Lambda(Q_n,R),
\qquad n=\alpha+\beta+\gamma,
\]
be a cocycle, where $r$ is multihomogeneous. Then, in
\(\operatorname{HH}^*(R)\),
\[
        D_i\bigl([F]\bigr)
        =\bigl(\deg_{x_i}(r)+\alpha_i\bigr)\,[F],
        \qquad
        (\alpha_0,\alpha_1,\alpha_2)=(\alpha,\beta,\gamma),
\]
where the coefficient is regarded as an element of the field \(K\).
\end{lemma}

\begin{proof}
Under the identification
\(\operatorname{Hom}_{\Lambda}(\Bbar_1(R),R)\cong
\operatorname{Hom}_K(\overline R,R)\), the cochain \(y_i\circ\Psi_1\)
is a \(K\)-linear map \(\overline R\to R\). Direct application of
the formulas of Proposition~\ref{prop:psi1} and the \(\Lambda\)-linearity of
\(y_i\) to the basis of \(\overline R\) gives \(y_i\Psi_1(\overline u)
=\partial_i(u)\) for all \(u\in B\setminus\{1\}\); for example, for
\(i=0\),
\[
        y_0\Psi_1(\overline a)=x_0x_1=a,
        \quad
        y_0\Psi_1(\overline b)=0,
        \quad
        y_0\Psi_1(\overline c)=x_2x_0=c,
        \quad
        y_0\Psi_1(\overline w)=x_0b=w,
\]
which agrees with \(\partial_0\). The cases \(i=1,2\) follow by cyclic permutation; thus
\(y_i\circ\Psi_1=\partial_i\).

Set \(f=F\circ\Psi_n\); this bar cochain represents the class \([F]\), and
the class \(y_i\) is represented by \(\partial_i\), so
\[
        D_i([F])
        =\bigl[[\partial_i,f]\circ\Phi_n\bigr].
\]
For the derivation \(\partial_i\) of degree \(1\), the Gerstenhaber bracket
with the bar cochain \(f\) of degree \(n\), when \(\Char K=2\), takes the form
\[
        [\partial_i,f](u_1,\ldots,u_n)
        =\partial_i\bigl(f(u_1,\ldots,u_n)\bigr)
        +\sum_{j=1}^{n}
        f(u_1,\ldots,\partial_i(u_j),\ldots,u_n).
\]

We compute \(G_i:=[\partial_i,f]\circ\Phi_n\) on \(e_{pqs}\),
\((p_0,p_1,p_2)=(p,q,s)\).

Since \(\Psi_n\Phi_n=\id\)
(Corollary~\ref{cor:Psi-Phi-identity}), we have
\(f\Phi_n(e_{pqs})=F(e_{pqs})\), so the first term gives
\(\partial_i(F(e_{pqs}))\).

By Proposition~\ref{prop:phi-paths},
each tensor in the expansion of $\Phi_n(e_{pqs})$
contains exactly $p_i$ factors $x_i$.
Since $\partial_i(x_j)=\delta_{ij}x_j$,
the second term is
$p_i f\bigl(\Phi_n(e_{pqs})\bigr)
=p_i F(e_{pqs})$. Consequently,
\[
        G_i(e_{pqs})=\partial_i\bigl(F(e_{pqs})\bigr)+p_i\,F(e_{pqs}).
\]

For \(F=[r]_{\alpha,\beta,\gamma}\), the right-hand side vanishes when
\((p,q,s)\ne(\alpha,\beta,\gamma)\), whereas for
\((p,q,s)=(\alpha,\beta,\gamma)\), using \(p_i=\alpha_i\) and
\(\partial_i(r)=\deg_{x_i}(r)\,r\), it equals
\(\bigl(\deg_{x_i}(r)+\alpha_i\bigr)r\). Thus,
\(G_i=\bigl(\deg_{x_i}(r)+\alpha_i\bigr)[r]_{\alpha,\beta,\gamma}\), and
passing to cohomology gives the required equality.
\end{proof}

Lemma~\ref{lem:Euler-revised} immediately gives
\begin{equation}\label{eq:Di-table-revised}
\begin{aligned}
D_i(h)&=h,
&\qquad
D_i(y_j)&=0,
&\qquad
D_i(p)&=p,\\
D_i(h_j)&=
\begin{cases}0,&j=i,\\ h_j,&j\ne i,\end{cases}
&\qquad
D_i(z_j^{(m)})&=
\begin{cases}m\,z_i^{(m)},&j=i,\\ z_j^{(m)},&j\ne i.\end{cases}
\end{aligned}
\end{equation}
For $1<s=\operatorname{ord}(\rho)<\infty$,
the order $s$ is odd, since $\Char K=2$.
Hence $s=1$ and $s-1=0$ in $K$, and in addition
\begin{equation}\label{eq:Di-finite-revised}
D_i(p_j)=
\begin{cases}p_j,&j=i,\\ 0,&j\ne i,\end{cases}
\qquad
D_i(t_j)=
\begin{cases}t_j,&j=i,\\ 0,&j\ne i.\end{cases}
\end{equation}

\subsection{Weights and selection rules}

The computation of \(\Delta\) on the generators relies on an additional
\(\mathbb Z^3\)-grading of the cochain space. Its role is as follows: \(\Delta\) and the Gerstenhaber bracket preserve this grading; if the target group contains no nonzero classes of the required weight, then the corresponding value is zero. This avoids direct computations in the bar complex and provides the main tool for the next two subsections.

\begin{definition}\label{def:weight-cochain-Q}
For a multihomogeneous \(r\in R\), the weight of the elementary cochain
\([r]_{\alpha,\beta,\gamma}\) is defined to be the vector
\[
        \operatorname{wt}\bigl([r]_{\alpha,\beta,\gamma}\bigr)
        =\deg_x(r)-(\alpha,\beta,\gamma)\in\mathbb Z^3.
\]
A cochain \(F\) of degree \(m\) is called \emph{homogeneous of weight}
\(\lambda\in\mathbb Z^3\) if
\[
        F(e_{\alpha,\beta,\gamma})\in R_{(\alpha,\beta,\gamma)+\lambda}
        \qquad\text{for all }\ \alpha+\beta+\gamma=m,
\]
where we set $R_d=0$ if at least one coordinate of $d$ is negative.
\end{definition}

Decomposing each value $F(e_{\alpha,\beta,\gamma})$ into
multihomogeneous components and then grouping the terms by
weight gives the weight decomposition
\begin{equation}\label{eq:weight-decomposition}
        \operatorname{Hom}_{\Lambda}(Q_m,R)
        =\bigoplus_{\lambda\in\mathbb Z^3}C^m_\lambda ,
\end{equation}
where $C^m_\lambda$ is the space of cochains of degree $m$ homogeneous
of weight $\lambda$.

By Lemma~\ref{lem:weights-revised}\textup{(i)}, the differential $\delta_Q$
preserves weight, so $C^{\bullet}_\lambda$ is a subcomplex of
$\operatorname{Hom}_\Lambda(Q_\bullet,R)$,
and~\eqref{eq:weight-decomposition} is a decomposition of complexes. Set
\[
        \HH^m_\lambda(R)=H^m\bigl(C^{\bullet}_\lambda\bigr);
\]
then
\begin{equation}\label{eq:HH-weight-decomposition}
        \HH^m(R)=\bigoplus_{\lambda\in\mathbb Z^3}\HH^m_\lambda(R),
\end{equation}
and every class decomposes uniquely into a finite sum of components homogeneous with respect to weight.

\begin{definition}\label{def:weight-cochain-Bar}
Let \(\lambda\in\mathbb Z^3\). A bar cochain
\(f\in\overline C^m(R)\) is called \emph{homogeneous of weight}
\(\lambda\) if, for any multihomogeneous arguments
\(u_1,\ldots,u_m\), either \(f(u_1,\ldots,u_m)=0\) or
\[
\deg_x f(u_1,\ldots,u_m)
=\deg_x(u_1)+\cdots+\deg_x(u_m)+\lambda.
\]
\end{definition}

Since the comparison morphisms $\Phi_*$ and $\Psi_*$
preserve weight, they induce isomorphisms between
the corresponding weight components of cohomology.

\begin{lemma}\label{lem:weights-revised}
Let \(f\) and \(g\) be cochains homogeneous with respect to weight. Then:
\begin{enumerate}[label=\textup{(\roman*)}]
\item \(\operatorname{wt}(\delta_Q f)=\operatorname{wt}(f)\);
\item \(\operatorname{wt}(f\smile g)=\operatorname{wt}(f)+\operatorname{wt}(g)\);
\item \(\operatorname{wt}([f,g])=\operatorname{wt}(f)+\operatorname{wt}(g)\);
\item \(\operatorname{wt}(\Delta f)=\operatorname{wt}(f)\);
\item $\Delta D_i=D_i\Delta$ on $\HH^*(R)$
for $i=0,1,2$.
\end{enumerate}
\end{lemma}
\begin{proof}

\emph{Part \textup{(i)}.}
Let $f\in C^m_\lambda$.
By the definition of weight,
\[
f(e_{\alpha,\beta,\gamma})
\in R_{(\alpha,\beta,\gamma)+\lambda},
\qquad \alpha+\beta+\gamma=m.
\]
Since $x_i\in R_{e_i}$,
formula~\eqref{eq:cochain-differential-revised}
gives
\[
(\delta_Q f)(e_{\alpha,\beta,\gamma})
\in R_{(\alpha,\beta,\gamma)+\lambda},
\qquad \alpha+\beta+\gamma=m+1.
\]
Hence
$\delta_Q(C^m_\lambda)\subseteq C^{m+1}_\lambda$.

\medskip
\emph{Part \textup{(ii)}.}
Let
\[
f\in C^m_\lambda,
\qquad
g\in C^n_\mu.
\]
Pass to the bar cochains $f\circ\Psi_m$ and $g\circ\Psi_n$,
denoting them by the same letters in this computation.
Since $\Psi_*$ preserves the internal multidegree,
their weights are $\lambda$ and $\mu$, respectively.

Let $u_1,\ldots,u_{m+n}\in\overline R$
be multihomogeneous with total multidegree $d$.
Set
\[
d'=\deg_x(u_1)+\cdots+\deg_x(u_m),
\qquad
d''=\deg_x(u_{m+1})+\cdots+\deg_x(u_{m+n}).
\]
Then $d'+d''=d$.
By the definition of the cup product,
\[
(f\smile g)(u_1,\ldots,u_{m+n})
=f(u_1,\ldots,u_m)\,g(u_{m+1},\ldots,u_{m+n}).
\]
Since
\[
f(u_1,\ldots,u_m)\in R_{d'+\lambda},
\qquad
g(u_{m+1},\ldots,u_{m+n})\in R_{d''+\mu},
\]
we obtain
\[
(f\smile g)(u_1,\ldots,u_{m+n})
\in R_{d'+\lambda}R_{d''+\mu}
\subseteq R_{d+\lambda+\mu}.
\]
Hence the bar cochain $f\smile g$ is homogeneous
of weight $\lambda+\mu$. Since $\Phi_*$
preserves the internal multidegree,
\[
(f\smile g)\circ\Phi_{m+n}
\in C^{m+n}_{\lambda+\mu}.
\]
Thus the transferred cup product
has weight $\lambda+\mu$.

\medskip
\emph{Part \textup{(iii)}.}
Let
\[
f\in C^m_\lambda,
\qquad
g\in C^n_\mu.
\]
If $m=n=0$, the bracket is zero.
Henceforth assume that $m+n\ge1$.

Pass to the bar cochains $f\circ\Psi_m$ and $g\circ\Psi_n$,
denoting them by the same letters in this computation.

First suppose that $m,n\ge1$.
For $1\le q\le m$, the partial insertion is given by
\[
(f\circ_q g)(u_1,\ldots,u_{m+n-1})
=f\bigl(u_1,\ldots,u_{q-1},g(u_q,\ldots,u_{q+n-1}),u_{q+n},\ldots,u_{m+n-1}\bigr).
\]
Let $u_1,\ldots,u_{m+n-1}\in\overline R$
be multihomogeneous with total multidegree $d$. Set
\[
d'=\deg_x(u_q)+\cdots+\deg_x(u_{q+n-1}).
\]
By the homogeneity of $g$,
\[
g(u_q,\ldots,u_{q+n-1})\in R_{d'+\mu},
\]
so
\[
(f\circ_q g)(u_1,\ldots,u_{m+n-1})
\in R_{(d-d')+(d'+\mu)+\lambda}
=R_{d+\lambda+\mu}.
\]
Hence each insertion $f\circ_q g$
is homogeneous of weight $\lambda+\mu$.
The same holds for the insertions $g\circ_q f$.

If $n=0<m$, the bar cochain $g$ is identified
with an element of $R_\mu$, and the same computation
applies with $d'=0$.
The case $m=0<n$ is analogous.
Thus the bracket
\[
[f,g]=f\circ g+g\circ f,
\qquad
f\circ g=\sum_{q=1}^{m}f\circ_q g,
\]
is homogeneous of weight $\lambda+\mu$.
Since $\Phi_*$ preserves the internal multidegree,
\[
[f,g]\circ\Phi_{m+n-1}\in C^{m+n-1}_{\lambda+\mu},
\]
which proves the assertion.

\medskip
\emph{Part \textup{(iv)}.}
Let $f\in C^m_\lambda$, $m\ge1$.
As in part~\textup{(ii)}, pass to the bar cochain
$f\circ\Psi_m$, denoting it in this computation
by the same letter.

Let $u_1,\ldots,u_{m-1}\in\overline R$
be multihomogeneous. Set
\[
d=\deg_x(u_1)+\cdots+\deg_x(u_{m-1}).
\]
For any $v\in B$, Tradler's formula gives
\[
\bigl\langle\Delta(f)(u_1,\ldots,u_{m-1}),v\bigr\rangle
=\sum_{r=1}^{m}
\bigl\langle
f(u_r,\ldots,u_{m-1},v,u_1,\ldots,u_{r-1}),1
\bigr\rangle.
\]
Since $f$ has weight $\lambda$ and
$\langle z,1\rangle=\tau(z)$ vanishes
on all multihomogeneous components except
$R_{(1,1,1)}$, the right-hand side can be nonzero
only if
\[
d+\deg_x(v)+\lambda=(1,1,1).
\]
On the other hand, for the dual basis,
\[
\deg_x(v^\vee)=(1,1,1)-\deg_x(v).
\]
Hence, by formula~\eqref{eq:dual-expansion-revised},
\[
\Delta(f)(u_1,\ldots,u_{m-1})
=\sum_{v\in B}
\bigl\langle\Delta(f)(u_1,\ldots,u_{m-1}),v\bigr\rangle
v^\vee
\in R_{d+\lambda}.
\]
Thus the bar cochain $\Delta(f)$ is homogeneous
of weight $\lambda$. Since $\Phi_*$ preserves
the internal multidegree,
\[
\Delta(f)\circ\Phi_{m-1}\in C^{m-1}_\lambda.
\]
For $m=0$, this holds since $\Delta=0$ by definition.

\medskip
\emph{Part \textup{(v)}.}
Let $u\in\HH^m_\lambda(R)$, where
$\lambda=(\lambda_0,\lambda_1,\lambda_2)$.

The proof of Lemma~\ref{lem:Euler-revised} gives
\[
D_i(u)=\lambda_i u,\qquad i=0,1,2.
\]

Since, by part~\textup{(iv)},
the operator $\Delta$ preserves weight, we have
\[
D_i(\Delta u)=\lambda_i\Delta u.
\]
Hence
\[
\Delta D_i(u)
=\Delta(\lambda_i u)
=\lambda_i\Delta u
=D_i(\Delta u).
\]

By the decomposition
\[
\HH^m(R)=\bigoplus_{\lambda\in\mathbb Z^3}
\HH^m_\lambda(R)
\]
and the linearity of $D_i$ and $\Delta$,
the resulting equality holds for all
$u\in\HH^m(R)$. Thus
\[
\Delta D_i=D_i\Delta,
\qquad i=0,1,2.
\]

\end{proof}

\begin{corollary}[Selection rule]\label{cor:selection-revised}
Let $u,v$ be classes homogeneous with respect to weight. If
$\HH^{|u|-1}_{\operatorname{wt}(u)}(R)=0$, then $\Delta(u)=0$; if
$\HH^{|u|+|v|-1}_{\operatorname{wt}(u)+\operatorname{wt}(v)}(R)=0$, then
$[u,v]=0$.
\end{corollary}

To apply the selection rule, we use the additive basis given above.

\subsection{Values of \(\Delta\) on the generators}

The values of $\Delta$ on the generators fall into three cases. The classes
$h$, $h_i$ have degree $0$, and $\Delta$ lowers degree, so
$\Delta(h)=\Delta(h_i)=0$. The classes $y_i$ have degree $1$ and weight $0$,
with $\HH^0_0(R)\ne0$, so the selection rule does not apply;
a direct computation gives $\Delta(y_i)=1$. For every other
generator $u$, we have $\HH^{|u|-1}_{\operatorname{wt}(u)}(R)=0$, and
hence $\Delta(u)=0$ by Corollary~\ref{cor:selection-revised}.

\begin{proposition}\label{prop:Delta-generators-revised}
On the generators of the algebra \(\operatorname{HH}^*(R)\), we have
\[
        \Delta(y_i)=1,
        \qquad
        \Delta(h)=\Delta(h_i)=\Delta(p)=\Delta(z_i^{(m)})=0,
\]
and, when \(1<s=\operatorname{ord}(\rho)<\infty\), additionally
\[
        \Delta(p_i)=\Delta(t_i)=0.
\]
\end{proposition}

\begin{proof}
\emph{Classes of degree \(0\).} The operator \(\Delta\) lowers degree,
so \(\Delta=0\) on \(\operatorname{HH}^0(R)\); in particular,
\(\Delta(h)=\Delta(h_i)=0\).

\medskip
\emph{The values $\Delta(y_i)$.}
By Lemma~\ref{lem:Euler-revised}, the class $y_i$ is represented
by the Euler derivation $\partial_i$.
Since $\partial_i(w)=w$ and
$\tau(\partial_i(u))=0$ for $u\in B\setminus\{w\}$,
we have $\tau\circ\partial_i=\tau$.
By Tradler's formula, for any $u\in R$,
\[
\langle\Delta(y_i),u\rangle
=\tau(\partial_i(u))
=\tau(u)
=\langle1,u\rangle.
\]
The nondegeneracy of the form gives
$\Delta(y_i)=1$ for $i=0,1,2$.
    
\medskip
\emph{Vanishing by the selection rule.} The class $\Delta(u)$ belongs to $\HH^{|u|-1}(R)$ and, by
Lemma~\ref{lem:weights-revised}\textup{(iv)}, has weight
$\operatorname{wt}(u)$. For $\lambda=(\lambda_0,\lambda_1,\lambda_2)\in\mathbb Z^3$,
set $|\lambda|=\lambda_0+\lambda_1+\lambda_2$.
The space $C^n_\lambda$ is spanned by
the elementary cochains $[r]_{\alpha,\beta,\gamma}$ with
$\alpha+\beta+\gamma=n$ and multihomogeneous $r$ satisfying
\[
        \deg_x(r)=\lambda+(\alpha,\beta,\gamma),
        \qquad
        |{\deg_x(r)}|=|\lambda|+n .
\]
Since $R_{d}\ne0$ only for $d\in\{0,1\}^3$, all
coordinates of $\deg_x(r)$ and $(\alpha,\beta,\gamma)$ are nonnegative. If,
for given $\lambda$ and $n$, no such $r$ and $(\alpha,\beta,\gamma)$
exist, then $C^n_\lambda=0$, and hence
$\HH^n_\lambda(R)=0$, which gives $\Delta(u)=0$ by
Corollary~\ref{cor:selection-revised}.

We verify this for the remaining
generators; by cyclic symmetry, it suffices to consider \(i=0\).
\[
\begin{array}{c|c|c}
        u & |u|-1 & \operatorname{wt}(u)\\ \hline
        p        & 2      & (-1,-1,-1)\\
        z_0^{(m)}& m-1    & (-m,1,1)\\
        p_0      & s-1    & (-s,0,0)\\
        t_0      & 2s-3   & (1,1-s,1-s)
\end{array}
\]
Consider the following cases, verifying in each
that there are no elementary cochains
$[r]_{\alpha,\beta,\gamma}$ of the required degree and weight.

\begin{enumerate}[label=\textup{(\alph*)}]
\item \emph{The classes \(p\) and \(p_0\).} The total \(x\)-degree of a possible
\(r\) is
\[
        |{\deg_x(r)}|=|{\operatorname{wt}(u)}|+(|u|-1),
\]
which gives, for \(p\), \((-3)+2=-1\), and for \(p_0\), \((-s)+(s-1)=-1\).
A negative total degree is impossible.

\item \emph{The classes \(z_0^{(m)}\).} The total \(x\)-degree is
\((2-m)+(m-1)=1\), but the weight coordinates corresponding to \(x_1\) and \(x_2\)
require \(\deg_{x_1}(r)\ge1\) and \(\deg_{x_2}(r)\ge1\), which
is incompatible with total degree \(1\).

\item \emph{The classes \(t_0\).} The total \(x\)-degree is
\((3-2s)+(2s-3)=0\), that is, \(r\) is a scalar; the first
weight coordinate then gives the index
\(\alpha=\deg_{x_0}(r)-\operatorname{wt}(t_0)_0=0-1=-1<0\).
\end{enumerate}

In every case, the required weight component of
\(\operatorname{HH}^{|u|-1}(R)\) is zero, so Corollary~\ref{cor:selection-revised}
gives
\[
        \Delta(p)=\Delta(z_i^{(m)})=\Delta(p_i)=\Delta(t_i)=0.
        \qedhere
\]
\end{proof}

\subsection{Completeness of the Gerstenhaber bracket table}

The computed values of \(\Delta\) on the generators allow us to determine the Gerstenhaber
bracket on all pairs of generators. The bracket is expressed in terms of
\(\Delta\) by the formula
\begin{equation}\label{eq:B-revised}
        [u,v]=\Delta(uv)+\Delta(u)\,v+u\,\Delta(v),
\end{equation}
which holds in characteristic \(2\) for homogeneous classes \(u,v\). In
particular, for every homogeneous \(g\),
\[
        [g,g]=g\circ g+g\circ g=0.
\]
Thus, to describe the bracket completely, it suffices to compute
\([u,v]\) on unordered pairs of distinct generators.

To write the table compactly, introduce the notation
\[
        Y=\rho^{-2}(y_0y_1+y_0y_2+y_1y_2),
\]
and, in the case of finite order \(1<s<\infty\), additionally
\[
        Z_i=z_i^{(s-1)},
        \qquad
        U_i=\rho^{-2}(y_{i+1}+y_{i+2})p_i,
\]
\[
        V_i=\rho^{-2}p_{i+1}p_{i+2},
        \qquad
        W=\rho^{-2}p^{s-1}.
\]

The selection of potentially nonzero pairs relies on the selection
rule~\ref{cor:selection-revised}, so we need the degrees and weights of
the generators. Set $\mathbf{1}=(1,1,1)$. Then
\begin{equation}\label{eq:generator-weights-revised}
\begin{array}{c|c|c}
        g & |g| & \operatorname{wt}(g)\\ \hline
        h        & 0     & \mathbf1\\
        h_i      & 0     & \mathbf1-e_i\\
        y_i      & 1     & 0\\
        p        & 3     & -\mathbf1\\
        z_i^{(m)}& m     & \mathbf1-(m+1)e_i\\
        p_i      & s     & -s\,e_i\\
        t_i      & 2s-2  & e_i-(s-1)(\mathbf1-e_i)
\end{array}
\end{equation}
In particular, for \(Z_i=z_i^{(s-1)}\),
\[
        |Z_i|=s-1,
        \qquad
        \operatorname{wt}(Z_i)=\mathbf1-s\,e_i.
\]

There are many pairs of distinct generators, but almost all of them
are eliminated by three simple observations:
pairs involving $y_i$ have already been described by the operators $D_i$;
among the other pairs, a zero product in $\HH^*(R)$
implies a zero bracket; the remaining pairs are checked using
the selection rule~\ref{cor:selection-revised}.

This leaves the short list given below.

\begin{lemma}\label{lem:candidate-pairs-revised}
Let \(u,v\) be distinct generators other than \(y_0,y_1,y_2\).
Then \([u,v]=0\), except for the following types of pairs, for which
the unique possible target component is indicated:
\[
\begin{array}{c|c}
        \text{pair}&\text{target component}\\ \hline
        (h,p)&\langle y_0y_1,y_0y_2,y_1y_2\rangle_K\\
        (h,p_i)&KZ_i\\
        (p,Z_i)&\langle y_0p_i,y_1p_i,y_2p_i\rangle_K\\
        (p,t_i)&Kp_{i+1}p_{i+2}\\
        (p_i,Z_{i+1})&Kt_{i+2}\\
        (p_i,Z_{i+2})&Kt_{i+1}\\
        (p_i,t_i)&Kp^{s-1}\\
        (t_i,Z_i)&\langle y_0y_1,y_0y_2,y_1y_2\rangle_K\,p^{s-2}\\
        (t_{i+1},t_{i+2})&\langle y_0,y_1,y_2\rangle_K\,p_i\,p^{s-2}
\end{array}
\]
All rows except the first concern the case of finite order
\(1<s<\infty\).
\end{lemma}

\begin{proof}
\emph{Step 1: pairs excluded by degree.} If \(|u|=|v|=0\), then
\([u,v]\in\operatorname{HH}^{-1}(R)=0\). This excludes all pairs from
\(\{h,h_0,h_1,h_2\}\).

\smallskip
\emph{Step 2: pairs with zero product.} For generators
\(u,v\notin\{y_0,y_1,y_2\}\),
Proposition~\ref{prop:Delta-generators-revised} gives
\(\Delta(u)=\Delta(v)=0\), so formula~\eqref{eq:B-revised}
takes the form
\[
        [u,v]=\Delta(uv),
\]
and every pair with \(uv=0\) gives \([u,v]=0\). The lists of vanishing
products differ in the two cases.

\emph{The case $s=\infty$.}
The following products vanish in $\HH^*(R)$:
\[
        h h_i,\quad h z_i^{(m)},\quad
        h_i h_j,\quad h_i z_j^{(m)},\quad h_i p,\quad
        z_i^{(m_1)}z_j^{(m_2)},\quad z_i^{(m)}p .
\]
The generators \(p_i\), \(t_i\) are absent here, and the last relation
holds for all \(m\). Thus all pairs involving \(h_i\), all
pairs within the family \(z^{(m)}\), and all pairs \((p,z_i^{(m)})\) are eliminated;
only the pair \((h,p)\) remains from the original list.

\emph{The case $1<s<\infty$.}
The products listed above vanish in $\HH^*(R)$,
except for $z_i^{(m)}p$, as do the following additional products:
\[
       h t_i,\quad h_i p_j\ (i\ne j),\quad h_i t_j,
\]
\[
        t_iz_j^{(m)}\ (i\ne j),
        \qquad
        t_iz_i^{(m)}\ (m\ne s-1),
\]
\[
        z_i^{(m)}p\ (m\ne s-1),
        \qquad
        z_i^{(m)}p_j\ (i\ne j,\ m\ne s-1).
\]
Thus the family $z^{(m)}$ enters the subsequent discussion only through
$Z_i=z_i^{(s-1)}$, except for the pairs $(z_i^{(m)},p_i)$, which are not
covered by the monomial relations; these are considered in Step~3.

\smallskip
\emph{Step 3: the remaining pairs.} For the remaining pairs,
we compute the degree $n=|u|+|v|-1$ and weight
$\lambda=\operatorname{wt}(u)+\operatorname{wt}(v)$ from
Table~\eqref{eq:generator-weights-revised} and then check
whether the component $\HH^{n}_{\lambda}(R)$ is zero; the types of basis
cocycles are taken from
\cite{GK06}.

The pair \((h,p)\): degree \(2\), weight
\(\operatorname{wt}(h)+\operatorname{wt}(p)=\mathbf1-\mathbf1=0\).
We consider this case in detail; the others are checked in the same way.

The condition \(\operatorname{wt}([r]_{\alpha,\beta,\gamma})=0\) with
\(\alpha+\beta+\gamma=2\) means that
\[
        \deg_x(r)=(\alpha,\beta,\gamma),
\]
that is, \(r\) is an element of total \(x\)-degree \(2\).

The only such elements of the basis \(B\) are \(a\), \(b\), \(c\), giving the cochains
\[
        [a]_{1,1,0},
        \qquad
        [b]_{0,1,1},
        \qquad
        [c]_{1,0,1}.
\]

Their classes form a basis of the component and
coincide with \(y_0y_1\), \(y_1y_2\), \(y_0y_2\) up to
nonzero factors that are powers of \(\rho\).

The pair \((h,p_i)\): degree \(0+s-1=s-1\), weight
\(\mathbf1-s\,e_i\). These parameters correspond to the class \(Z_i=z_i^{(s-1)}\)
of degree \(s-1\), and the component is one-dimensional: \(K Z_i\).

The pair \((h_i,p_i)\): degree \(s-1\), weight \(\mathbf1-(s+1)e_i\); the total
\(x\)-degree of a representative is \(1\), but the weight coordinates corresponding to the two
other variables are \(1\), requiring \(\deg_{x_j}(r)\ge1\) for
both \(j\ne i\), which is incompatible with degree \(1\). Hence
\([h_i,p_i]=0\).

The pair \((p,Z_i)\): degree \(3+(s-1)-1=s+1\), weight \(-s\,e_i\);
the component is spanned by \(y_0p_i\), \(y_1p_i\), \(y_2p_i\).

The pair \((p,t_i)\): degree \(2s\), weight
\(-\mathbf1+e_i-(s-1)(\mathbf1-e_i)\), which equals the weight of
the product \(p_{i+1}p_{i+2}\); the component is one-dimensional.

The pair \((p_i,z_i^{(m)})\): the degree is \(m+s-1\), and the weight is
\(\mathbf1-(m+1+s)e_i\). The total \(x\)-degree of a possible
representative is \(1\), whereas the weight coordinates corresponding to the two
variables \(x_l\), \(l\ne i\), are \(1\) and require
\(\deg_{x_l}(r)\ge1\) for both such \(l\); this is incompatible, whence
\([p_i,z_i^{(m)}]=0\) for all \(m\).

The pair \((p_i,Z_j)\), \(j\ne i\): degree \(2s-2\), weight
\(\mathbf1-s\, e_i-s\, e_j\). The weight equation has
a unique nonnegative solution: \(r=x_l\) for \(l\ne i,j\), with
indices \(\alpha_i=\alpha_j=s-1\), \(\alpha_l=0\), corresponding to
the class \(t_l\). Thus the target components for the pairs \((p_i,Z_{i+1})\) and
\((p_i,Z_{i+2})\) are \(Kt_{i+2}\) and \(Kt_{i+1}\), respectively.

The pair \((p_i,t_j)\): a nonnegative solution exists only when
\(j=i\); the degree is \(3s-3\), and the weight equals that of \(p^{s-1}\).

The pair \((t_i,Z_i)\): degree \((2s-2)+(s-1)-1=3s-4\), weight
\(\operatorname{wt}(t_i)+\mathbf1-s\,e_i\). These parameters correspond to
the classes \(y_jy_k\,p^{\,s-2}\), and the component is spanned by
\(y_0y_1p^{\,s-2}\), \(y_0y_2p^{\,s-2}\), \(y_1y_2p^{\,s-2}\).

The pair \((t_{i+1},t_{i+2})\): degree \(2(2s-2)-1=4s-5\), and the sum of the weights
corresponds to the classes \(y_j\,p_i\,p^{\,s-2}\); the component is spanned by
\(y_0p_ip^{\,s-2}\), \(y_1p_ip^{\,s-2}\), \(y_2p_ip^{\,s-2}\).

The pair $(p,p_i)$: degree $s+2$, total weight
$-\mathbf1-s\, e_i$; the total $x$-degree of the value would be
$|\lambda|+n=(-3-s)+(s+2)=-1<0$, which is impossible.

The pair $(p_i,p_j)$, $i\ne j$: degree $2s-1$, total weight
$-s\, e_i-s\, e_j$; the total $x$-degree of the value is again
$(-2s)+(2s-1)=-1<0$.

In the last two cases, the weight components are zero, and
Corollary~\ref{cor:selection-revised} implies that the bracket vanishes;
the other pairs give the components indicated in the table. The enumeration is
therefore exhaustive.
\end{proof}

\begin{theorem}\label{prop:B-complete-revised}
The Gerstenhaber bracket on \(\operatorname{HH}^*(R)\) is completely
determined. On pairs involving \(y_i\), it is given by
Tables~\eqref{eq:Di-table-revised}
and~\eqref{eq:Di-finite-revised}. On the other pairs of generators:

\emph{if \(s=\infty\)}, the only nonzero value is
\[
        [h,p]=Y;
\]

\emph{if \(1<s<\infty\)}, the nonzero values are
\[
        [h,p]=Y
\]

\begin{equation}\label{eq:B-resonance-revised}
\begin{aligned}
        [h,p_i]&=Z_i,
        &\qquad
        [p,Z_i]&=U_i,
        &\qquad
        [p,t_i]&=V_i,\\
        [p_i,Z_{i+1}]&=t_{i+2},
        &\qquad
        [p_i,Z_{i+2}]&=t_{i+1},
        &\qquad
        [p_i,t_i]&=W .
\end{aligned}
\end{equation}
\end{theorem}

\begin{proof}
Since $\Delta$ vanishes on all generators
except $y_i$, formula~\eqref{eq:B-revised}
takes the following form for the pairs under consideration:
$[u,v]=\Delta(uv)$.
We use the presentation of the ring from
Theorem~\ref{thm:ring-presentation}.

\medskip
\emph{The case \(s=\infty\).} By
Lemma~\ref{lem:candidate-pairs-revised}, only the
pair \((h,p)\) remains. From \(\Delta(y_i)=1\) and \([y_i,y_j]=0\), the
second-order formula gives
\[
        \Delta(y_0y_1y_2)=y_0y_1+y_0y_2+y_1y_2,
\]
so the relation \(hp=\rho^{-2}y_0y_1y_2\) yields
\[
        [h,p]=\Delta(hp)=\rho^{-2}\Delta(y_0y_1y_2)=Y .
\]

\medskip
\emph{The case \(1<s<\infty\).} The value \([h,p]=Y\) follows from the same
computation. The other pairs are treated using the relations
\begin{equation}\label{eq:ring-inputs-revised}
\begin{gathered}
        hp_i=y_iZ_i,
        \qquad
        Z_ip=\rho^{-2}y_{i+1}y_{i+2}p_i,\\
        Z_{i+1}p_i=t_{i+2}y_i,
        \qquad
        Z_{i+2}p_i=t_{i+1}y_i,\\
        t_ip=\rho^{-2}y_ip_{i+1}p_{i+2},
        \qquad
        t_ip_i=\rho^{-2}y_ip^{s-1}.
\end{gathered}
\end{equation}

For the pair \((h,p_i)\),
\[
        \Delta(hp_i)=\Delta(y_iZ_i)=Z_i+D_i(Z_i)
        =Z_i+(s-1)Z_i=Z_i,
\]
since $s-1=0$ in $K$.

For the pair \((p,Z_i)\),
\[
        \Delta(Z_ip)=\rho^{-2}\Delta(y_{i+1}y_{i+2}p_i)
        =\rho^{-2}(y_{i+1}+y_{i+2})p_i=U_i,
\]
since \(D_{i+1}(p_i)=D_{i+2}(p_i)=0\)
by~\eqref{eq:Di-finite-revised}.

For the pairs \((p_i,Z_{i+1})\) and \((p_i,Z_{i+2})\),
\[
        \Delta(Z_{i+1}p_i)=\Delta(t_{i+2}y_i)
        =t_{i+2}+D_i(t_{i+2})=t_{i+2},
\]
and, similarly, \(\Delta(Z_{i+2}p_i)=t_{i+1}\).

For the pair \((p,t_i)\),
\[
        \Delta(t_ip)=\rho^{-2}\Delta(y_ip_{i+1}p_{i+2})
        =\rho^{-2}p_{i+1}p_{i+2}=V_i .
\]

For the pair \((p_i,t_i)\),
\[
        \Delta(t_ip_i)=\rho^{-2}\Delta(y_ip^{s-1})
        =\rho^{-2}\bigl(p^{s-1}+(s-1)[y_i,p]\,p^{s-2}\bigr)
        =\rho^{-2}p^{s-1}=W,
\]
since the contribution of the pairs \((y_i,p)\) is \((s-1)p^{s-1}=0\).

\medskip
\emph{Vanishing of the remaining candidates.} By
Lemma~\ref{lem:candidate-pairs-revised}, among pairs outside the displayed list,
the selection rule fails to exclude only \((t_i,Z_i)\) and
\((t_{i+1},t_{i+2})\). From \(t_iZ_i=\rho^{-4}y_0y_1y_2p^{\,s-2}\), with
\(r=s-2\), we obtain
\[
        \Delta\bigl(y_0y_1y_2p^{\,r}\bigr)
        =(y_0y_1+y_0y_2+y_1y_2)p^{\,r}
        +r\,(y_0y_1+y_0y_2+y_1y_2)p^{\,r}=0,
\]
since \(s\) is odd and hence \(r=s-2=1\) in \(K\); therefore
\([t_i,Z_i]=0\). Similarly, from
\(t_{i+1}t_{i+2}=\rho^{-4}y_{i+1}y_{i+2}p_ip^{\,s-2}\), we obtain
\[
        \Delta\bigl(y_{i+1}y_{i+2}p_ip^{\,s-2}\bigr)
        =(y_{i+1}+y_{i+2})p_ip^{\,s-2}
        +(s-2)(y_{i+1}+y_{i+2})p_ip^{\,s-2}=0,
\]
where the contribution of the pairs \((y_{i+1},p_i)\), \((y_{i+2},p_i)\) is zero
by~\eqref{eq:Di-finite-revised}; hence \([t_{i+1},t_{i+2}]=0\).

All other pairs are excluded by Lemma~\ref{lem:candidate-pairs-revised}.
\end{proof}
By the Leibniz rule and bilinearity, the values of the bracket
on the generators determine it on the whole algebra.
Thus Theorems~\ref{thm:ring-presentation}
and~\ref{prop:B-complete-revised} completely describe
the Gerstenhaber algebra structure on $\HH^*(R)$.

\subsection{The second-order formula}

The operator \(\Delta\) is not a derivation, so its
values on products cannot be recovered from its values on
the factors by the Leibniz rule. Its deviation from being a derivation
is measured precisely by the Gerstenhaber bracket: the defining relation
of the BV structure when \(\Char K=2\) gives
\begin{equation}\label{eq:second-order-N2}
        \Delta(uv)=\Delta(u)\,v+u\,\Delta(v)+[u,v].
\end{equation}
The following proposition extends~\eqref{eq:second-order-N2} to
a product of any number of factors; together with Proposition~\ref{prop:Delta-generators-revised} and
Theorem~\ref{prop:B-complete-revised}, it allows us to compute
\(\Delta\) on all products of generators.

\begin{proposition}\label{prop:second-order-revised}
Let \(g_1,\ldots,g_N\) be homogeneous elements of
\(\operatorname{HH}^*(R)\). Then
\begin{equation}\label{eq:second-order-revised}
        \Delta(g_1\cdots g_N)
        =\sum_{r=1}^{N}\Delta(g_r)\prod_{\ell\ne r}g_\ell
        +\sum_{1\le r<q\le N}[g_r,g_q]\prod_{\ell\ne r,q}g_\ell .
\end{equation}
\end{proposition}

\begin{proof}
Induction on \(N\); for \(N=2\), the formula coincides
with~\eqref{eq:second-order-N2}. The induction step follows by substituting
\(u=g_1\cdots g_{N-1}\), \(v=g_N\) into~\eqref{eq:second-order-N2} and
expanding \(\Delta(u)\) by the induction hypothesis and
\([u,g_N]\) by the Leibniz rule
for the bracket.
\end{proof}

\subsection{The full formula in the case of infinite order}

Let \(\operatorname{ord}(\rho)=\infty\), and let
\begin{equation}\label{eq:normal-monomial-infinite}
        M=h^{\varepsilon}
        \prod_{j=0}^{2}h_j^{\eta_j}
        \prod_{j=0}^{2}y_j^{\alpha_j}\,
        p^{\beta}
        \prod_{j=0}^{2}\prod_{m\ge1}
        \bigl(z_j^{(m)}\bigr)^{\gamma_{j,m}}
\end{equation}
be an arbitrary product of generators of $\HH^*(R)$ with nonnegative integer exponents, only finitely many of
the $\gamma_{j,m}$ being nonzero. Set
\begin{equation}\label{eq:omega-infinite}
        \omega^{\infty}_i(M)
        =\varepsilon+\sum_{j\ne i}\eta_j+\beta
        +\sum_{m\ge1}m\,\gamma_{i,m}
        +\sum_{j\ne i}\sum_{m\ge1}\gamma_{j,m} .
\end{equation}

\begin{theorem}\label{thm:BV-infinite-revised}
Let \(\operatorname{ord}(\rho)=\infty\). For every product
of generators \(M\) of the form~\eqref{eq:normal-monomial-infinite},
\begin{equation}\label{eq:BV-infinite-revised}
        \Delta(M)
        =\sum_{i=0}^{2}
        \alpha_i\bigl(1+\omega^{\infty}_i(M)\bigr)\,
        \frac{M}{y_i}
        +\varepsilon\beta\,Y\,\frac{M}{hp},
\end{equation}
where the coefficients are taken in $K$.
Here and below, the notation $M/g$ means deleting
the factors constituting $g$ from the chosen
product $M$; terms for which
the corresponding factors are absent are omitted.
\end{theorem}

\begin{proof}
Apply formula~\eqref{eq:second-order-revised} to the expression of \(M\) as
a product of generators.

In the first sum, only the factors \(y_i\) contribute nonzero terms; there are
\(\alpha_i\) of them, each contributing \(M/y_i\). The pairwise contributions involving \(y_i\)
are \([y_i,g]=D_i(g)\); by Table~\eqref{eq:Di-table-revised},
each factor \(g\) of the product \(M/y_i\) is an eigenvector
of the operator \(D_i\), and the sum of the corresponding eigenvalues,
counted with the multiplicities of the factors, is precisely
\(\omega^{\infty}_i(M)\). Together these give the first term of
formula~\eqref{eq:BV-infinite-revised}. Pairs of identical generators
make no contribution, since \([g,g]=0\).

Among pairs of distinct generators not involving \(y_i\),
Theorem~\ref{prop:B-complete-revised} shows that the only nonzero value is that of
the pair \((h,p)\), namely \(Y\); the number of ways to choose it in
the product \(M\) is \(\varepsilon\beta\), giving the second
term.
\end{proof}

\subsection{The full formula in the case of finite order}

Let \(1<s=\operatorname{ord}(\rho)<\infty\), and let
\begin{equation}\label{eq:normal-monomial-finite}
        M=h^{\varepsilon}
        \prod_{j=0}^{2}h_j^{\eta_j}
        \prod_{j=0}^{2}y_j^{\alpha_j}\,
        p^{\beta}
        \prod_{j=0}^{2}p_j^{\pi_j}
        \prod_{j=0}^{2}t_j^{\vartheta_j}
        \prod_{j=0}^{2}\prod_{m=1}^{s-1}
        \bigl(z_j^{(m)}\bigr)^{\gamma_{j,m}}
\end{equation}
be an arbitrary product of generators of $\HH^*(R)$ with nonnegative integer exponents.
Set $\zeta_i=\gamma_{i,s-1}$, the exponent
with which $Z_i=z_i^{(s-1)}$ occurs in $M$, and
\begin{equation}\label{eq:omega-finite-revised}
        \omega^{s}_i(M)
        =\varepsilon+\sum_{j\ne i}\eta_j
        +\beta+\pi_i+\vartheta_i
        +\sum_{m=1}^{s-1}m\,\gamma_{i,m}
        +\sum_{j\ne i}\sum_{m=1}^{s-1}\gamma_{j,m} .
\end{equation}

\begin{theorem}\label{thm:BV-finite-revised}
Let \(1<s=\operatorname{ord}(\rho)<\infty\). For every
product of generators \(M\) of the form~\eqref{eq:normal-monomial-finite},
\begin{equation}\label{eq:BV-finite-revised}
\begin{aligned}
        \Delta(M)
        ={}&\sum_{i=0}^{2}
        \alpha_i\bigl(1+\omega^{s}_i(M)\bigr)\,\frac{M}{y_i}
        +\varepsilon\beta\,Y\,\frac{M}{hp}
        +\sum_{i=0}^{2}\varepsilon\pi_i\,Z_i\,\frac{M}{hp_i}\\
        &+\sum_{i=0}^{2}\beta\zeta_i\,U_i\,\frac{M}{pZ_i}
        +\sum_{i=0}^{2}\beta\vartheta_i\,V_i\,\frac{M}{pt_i}\\
        &+\sum_{i=0}^{2}\pi_i\zeta_{i+1}\,t_{i+2}\,
        \frac{M}{p_iZ_{i+1}}
        +\sum_{i=0}^{2}\pi_i\zeta_{i+2}\,t_{i+1}\,
        \frac{M}{p_iZ_{i+2}}\\
        &+\sum_{i=0}^{2}\pi_i\vartheta_i\,W\,\frac{M}{p_it_i},
\end{aligned}
\end{equation}
where the coefficients are taken in \(K\), and
terms are omitted whenever a factor in the denominator is absent from the chosen product. Since
products of generators span \(\operatorname{HH}^*(R)\) as
a \(K\)-vector space, this computes the operator \(\Delta\) on the whole
algebra.
\end{theorem}

\begin{proof}
Repeat the argument in the proof of Theorem~\ref{thm:BV-infinite-revised}, using Proposition~\ref{prop:second-order-revised} and taking into account the additional values of \(D_i\) and of the bracket from Theorem~\ref{prop:B-complete-revised}.
\end{proof}

\subsection{The boundary case \(s=1\)}\label{sec:s1}

When \(s=1\), that is, \(\rho=1\), the relations take the form
\(x_{i+1}x_i=x_ix_{i+1}\), the algebra \(R\) is commutative, and
\[
        R\cong K[x_0,x_1,x_2]/(x_0^2,x_1^2,x_2^2)
        \cong K[(C_2)^3]\cong K[C_2]^{\otimes3}.
\]
The corresponding cohomology algebra is the one in case~II.1
of~\cite[Theorem~2.1]{GK06}, that is, \(\HH^*(R)\cong A_6\); in
particular, the family \(z_i^{(m)}\) is absent here, the classes \(t_i\)
coincide with \(u_i\), and the set of generators is
\(\{u_i,p_i\}_{i\in\{0,1,2\}}\). Therefore
the formulas of Theorem~\ref{thm:BV-finite-revised}, which assumes \(s>1\), do not
apply to this case, and it is considered separately.

The BV structure in this case was obtained
in~\cite[Theorem~4.3]{LZ14} by reduction to the one-variable case using the isomorphism of BV algebras
\[
\HH^*(A\otimes B)\cong\HH^*(A)\otimes\HH^*(B).
\]
In our case, we have
\[
        \HH^*(R)
        \cong K[u_0,u_1,u_2,p_0,p_1,p_2]/(u_0^2,u_1^2,u_2^2),
        \qquad |u_i|=0,\ |p_i|=1,
\]
where \(u_i=[x_i]_{0,0,0}\) and \(p_i=[1]_{e_i}\).

The operator \(\Delta\) is given on monomials
\[
        M=\prod_{i=0}^{2}u_i^{\varepsilon_i}p_i^{a_i},
        \qquad \varepsilon_i\in\{0,1\},\ a_i\ge0,
\]
by the formula
\begin{equation}\label{eq:BV-s1}
        \Delta(M)=\sum_{i=0}^{2}\varepsilon_ia_i\,\frac{M}{u_ip_i}
\end{equation} where the coefficients are taken in \(K\), and
terms with \(\varepsilon_i=0\) or \(a_i=0\) are omitted. Since
\(\Char K=2\), a term is nonzero precisely when
\(\varepsilon_i=1\) and \(a_i\) is odd. In particular,
\[
        \Delta(u_i)=\Delta(p_i)=0,
        \qquad
        \Delta(u_ip_j)=\delta_{ij},
\]
and the Gerstenhaber bracket, recovered from~\eqref{eq:BV-s1} by
the defining relation of the BV structure, is determined by the Poisson rule and the relations
\[
        [u_i,u_j]=0,
        \qquad
        [p_i,p_j]=0,
        \qquad
        [u_i,p_j]=\delta_{ij}.
\]

\section*{Acknowledgments}

The author thanks A.~I.~Generalov for proposing the problem
of computing the BV structure on the Hochschild cohomology
of the Liu--Schulz algebras.

\end{document}